\documentclass[journal]{IEEEtran}

\ifCLASSINFOpdf
\else
\fi

\usepackage{multirow}
\usepackage{stfloats}

\usepackage{empheq}
\usepackage{mathrsfs}
\usepackage{amsfonts}
\usepackage{mathtools}
\usepackage{mathrsfs}
\usepackage{graphicx}          
\usepackage{amsmath}

\usepackage{enumerate,subcaption,graphicx}
\usepackage{algorithm,algpseudocode}
\algnewcommand{\algorithmicgoto}{\textbf{go to}}%
\algnewcommand{\Goto}[1]{\algorithmicgoto~\ref{#1}}%
\algnewcommand{\LineComment}[1]{\Statex \(\triangleright\) #1}
\algnewcommand{\LineCommentN}[1]{\Statex \hspace{1cm}\(\triangleright\) #1}
\usepackage{multirow}
\usepackage{stfloats}
\usepackage[dvipsnames]{xcolor}
\usepackage{amsthm}
\newtheoremstyle{mystyle}
  {}
  {}
  {\itshape}
  {}
  {\bfseries}
  {.}
  { }
  {\thmname{#1}\thmnumber{ #2}\thmnote{ (#3)}}
  
\graphicspath{{../Yong_Mohammad_TAC_MM/}}
\DeclareMathOperator*{\opt}{opt}

\theoremstyle{mystyle}
\usepackage{cite}
\usepackage{amsmath}
\usepackage{amssymb}
\newtheorem{thm}{Theorem}
\newtheorem{defn}{Definition}
\newtheorem{lem}{Lemma}

\newtheorem{cor}{Corollary}
\newtheorem{prop}{Proposition}
\newtheorem{exm}{Example}
\newtheorem{rem}{Remark}
\newtheorem{assumption}{Assumption}
\newtheorem{problem}{Problem}

\newcommand{\one}{1}

\definecolor{officegreen}{rgb}{0.0, 0.5, 0.0}
\newcommand{\moh}[1]{{\color{black} #1}}
\newcommand{\sy}[1]{{\color{black} #1}}
\newcommand{\syo}[1]{{\color{black} #1}}
\newcommand{\mo}[1]{{\color{black} #1}}
\begin{document}

%
\title{Tight Remainder-Form Decomposition Functions \\ with Applications to Constrained Reachability and Interval Observer Design}
%
%
%

\author{Mohammad~Khajenejad,~\IEEEmembership{Student Member,~IEEE,}
        and~Sze~Zheng~Yong,~\IEEEmembership{Member,~IEEE}
\thanks{M. Khajenejad and S.Z. Yong are with the School for Engineering of Matter, Transport and Energy, Arizona State University, Tempe, AZ, USA (e-mail: \{mkhajene, szyong\}@asu.edu).}}

\maketitle

\begin{abstract}
\sy{This paper proposes} a tractable family of remainder-\sy{form}  mixed-monotone decomposition functions \sy{that are useful for over-approximating the image set of nonlinear mappings in reachability and estimation problems. In particular, our approach applies to} 
\sy{a new} class \sy{of nonsmooth nonlinear systems that we call either-sided locally Lipschitz (ELLC) systems, which we show}  
to be a superset of locally Lipschitz continuous systems, \sy{thus expanding the set of systems that are formally known to be  mixed-monotone}. 
\sy{In addition, we derive} 
lower and upper bounds for the over-approximation error \sy{and show that} 
the lower bound is achievable \sy{with our proposed approach}.
 Moreover, we develop a set inversion algorithm that along with the proposed decomposition functions, can be used for constrained reachability analysis and interval observer design for continuous and discrete-time systems with bounded noise. 
\end{abstract}


\IEEEpeerreviewmaketitle

\section{Introduction}
\IEEEPARstart{M}onotonicity \sy{properties of systems have proven to be very powerful and useful for analyzing and controlling complex systems \cite{hirsch2006monotone,angeli2003monotone}. Building upon this idea, it \syo{was} 
shown that certain nonmonotone systems can be lifted to higher dimensional monotone systems that can be potentially used to deduce critical information about the original system  (see e.g., \cite{kulenovic2006global,enciso2006nonmonotone,gouze1994monotone}) by decomposing the system dynamics into increasing and decreasing components. Systems that are decomposable in this manner are called \emph{ mixed-monotone} and  are significantly more general than the class of monotone systems.}

\sy{Furthermore,  mixed-monotonicity has also proven to be very beneficial for system analysis and control.}
For instance, if  mixed-monotonicity holds, it can be concluded that the original system has global asymptotic stability by proving the nonexistence of equilibria of the 
\sy{lifted} system, except in a certain lower-dimensional subspace \cite{smith2008global,chu1998mixed}. 
\sy{Moreover,} 
forward invariant and attractive sets of the original system
can be identified \cite{abate2020computing} 
and reachable sets of the original system can be efficiently \sy{over-}approximated 
\sy{and used for} state estimation and abstraction-based control synthesis \cite{abate2020computing,coogan2016stability,coogan2015efficient,yang2019fuel}. 
\sy{However, the usefulness of these system lifting techniques for analysis and control is highly dependent on the \emph{tightness} of the decomposition approaches, their \emph{computational tractability} \cite{yang2019tight,abate2020tight} and their applicability to a broad class of systems; therefore, the capability to compute or construct \emph{tractable} and \emph{tight}  mixed-monotone decomposition functions for a broad range of nonlinear, uncertain and constrained systems is of great interest and will have a significant impact.}
\emph{Literature review}. \sy{\syo{M}ixed-monotone decomposition functions are generally \syo{not unique,} 
hence several seminal studies have addressed the issue of identifying and/or constructing appropriate decomposition functions with somewhat different yet highly related definitions and corresponding sufficient conditions for  mixed-monotonicity \cite{yang2019tight,yang2019sufficient,abate2020tight,meyer2019tira,meyer2019hierarchical,coogan2015efficient,coogan2016mixed,chu1998mixed}. In particular,} 
recent studies in \cite{yang2019tight,abate2020tight} 
provided \emph{tight} decomposition functions for unconstrained discrete\sy{-time} and continuous-time dynamical systems, respectively, \sy{whose \emph{computability}} 
rely on \moh{the} global solvability of nonlinear optimization programs, which is \sy{only} guaranteed in \syo{some} specific cases such as when the vector field is \emph{Jacobian sign-stable}, or when all \syo{its} 
\emph{\syo{stationary/}critical 
points} 
can be \syo{computed analytically.} 
On the other hand, 
\sy{an} interesting study in \cite{yang2019sufficient} 
proposed computable and constructive (but not necessarily tight) decomposition functions for differentiable vector fields with 
known bounds for the derivatives. \sy{Building on these frameworks, we aim to obtain computable/tractable and tighter decomposition functions for a broad class of nonsmooth systems.} 

\sy{Another relevant body of literature pertains to \emph{interval arithmetic} \cite{jaulinapplied,cornelius1984computing,moore2009introduction,alefeld2000interval,kieffer2002guaranteed}, which \syo{has} 
been successfully applied to problems in numerical analysis, set estimation, motion planning, etc. Specifically, \emph{inclusion functions} and variations thereof (e.g., natural,
centered-form and mixed-form inclusions) that are based on interval arithmetic can be directly related to decomposition functions and can similarly be used for over-approximation of the ranges/image sets of functions and for state observer designs \cite{yang2020accurate,alamo2005guaranteed,alamo2008set,scott2013bounds,shen2017rapid,kieffer2002guaranteed}. In addition, various properties of inclusion functions have been studied such as their convergence rates and the \emph{subdivision principle} \cite{cornelius1984computing}. Subsequent studies further introduced refinements of interval over-approximations and set inversion algorithms that can incorporate} 
%
new sources of information about the system such as state constraints, measurements/observations, manufactured redundant variables, second-order derivatives, etc., e.g., in \cite{jaulinapplied,yang2020accurate,alamo2005guaranteed,alamo2008set,scott2013bounds}, \sy{which can be also beneficial for decomposition function-based constrained reachability and set-valued estimation problems that we consider in this \mo{work}.}

\emph{Contribution.} \sy{In this paper, we introduce a class of  mixed-monotone decomposition functions whose construction is both computationally tractable and tight for a broad range of discrete-time and continuous-time nonsmooth nonlinear systems. The proposed  mixed-monotone decomposition functions are in the \emph{remainder form} (based on the terminology in \cite{cornelius1984computing}) and we also show that they result in \emph{difference of monotone functions}, which bear some resemblance with difference of convex functions in DC programming that is also widely used in optimization and state estimation problems, e.g., \cite{alamo2008set,horst1999dc}.

Our work contributes to the literature on  mixed-monotone decomposition functions and more generally, inclusion functions in multiple ways:
\begin{enumerate}[1)]
\item First, our proposed remainder-form decomposition functions apply to a new broader class of nonsmooth nonlinear systems, which we call \emph{either-sided locally Lipschitz continuous} (ELLC) systems, \mo{that} is proven in the paper to be a superset of locally Lipschitz continuous (LLC) systems. This new system class relaxes the (almost everywhere) smoothness and bounded gradient requirements and allows nonsmooth vector fields with only one-sided bounded \emph{Clarke 
Jacobians}. 
\item Our construction approach for remainder-form decomposition functions  is tractable (i.e., computable in closed-form with a non-iterative algorithm) 
and proven to be the tightest for this family of decomposition functions, which we also show to include \cite{yang2019sufficient}, thus generalizing and improving on the method in \cite{yang2019sufficient}.
\item  We further show that the over-approximation of the image sets of ELLC systems converges to the true tightest enclosing interval at least linearly, when the interval domain width goes to zero and \syo{moreover, a} 
\emph{subdivision principle} 
\syo{applies for improving} the enclosure of the range/image set.
\item We introduce a novel set inversion algorithm based on  mixed-monotone decomposition functions as an alternative to SIVIA \cite{jaulinapplied} and the refinement algorithm in \cite{yang2020accurate}, which enables the design of algorithms for constrained reachability and interval observers for systems with known constraints, modeling redundancy and/or  sensor measurements.
\end{enumerate}
Moreover, it is noteworthy that the inclusion function based on  mixed-monotone decomposition functions can be used alongside any existing inclusion functions, where 
 the ``best of them" \syo{(by virtue of an intersection property)} is chosen, 
since we observed that \mo{in general,} no single inclusion function consistently outperforms all others. Further\syo{more}, our proposed inclusion function can be directly and simply integrated in\syo{to} existing set inversion, constrained reachability and interval observer algorithms, e.g., \cite{jaulinapplied,yang2020accurate,khajenejad2020interval}, without much modifications. Finally, to demonstrate the effectiveness of the proposed algorithms for decomposition function construction and set inversion, we compare them with existing inclusion/decomposition functions in the literature. 
} 
\section{Background and Problem Formulation}
\subsection{Notation}
$\mathbb{N}$, \sy{$\mathbb{N}_a$,} $\mathbb{R}^{n_z}$ and $\mathbb{R}^{n \times m}$ 
denote the set of positive integers, \sy{the first $a$ positive integers,} the $n_z$-dimensional Euclidean space and the space of $n$ by $m$ real matrices, respectively. Moreover, $\forall z,\underline{z},\overline{z} \in \mathbb{R}^{n_z}$, $\underline{z} \leq \overline{z} \Leftrightarrow \underline{z}_i \leq \overline{z}_i ,\forall i \in \sy{\mathbb{N}_{n_z}}$, 
where ${z}_i$ denotes the $i$-th \sy{element of $z$.} 
Further, $\sy{\mathbb{I}}\mathcal{Z}\triangleq[\underline{z},\overline{z}] \triangleq \mo{\{}z\in \mathbb{R}^{n_z}| \underline{z} \leq z \leq \overline{z}\mo{\}}$ 
and \sy{$d(\mathbb{I}\mathcal{Z}) \triangleq \|\overline{z}-\underline{z}\|_{\infty}$} are called 
a closed interval/hyperrectangle in $\mathbb{R}^{n_z}$ and the diameter of $\sy{\mathbb{I}}\mathcal{Z}$, 
\sy{respectively}, where $\|z\|_{\infty} \triangleq \max\limits_i|z_i|$ denotes the $\ell_{\infty}$-norm of $z \in \mathbb{R}^{n_z}$. 
The set of all intervals in $\mathbb{R}^{n_z}$ is denoted by $\mathbb{IR}^{n_z}$. 

\subsection{Definitions and Preliminaries}
\sy{First,} we briefly introduce some of the main concepts that we use through\sy{out} the paper, as well as some important existing results that will be used  \sy{for comparisons and for deriving} 
our main results. 
We start by introducing inclusion \sy{and decomposition} functions and some of their 
\sy{typical} instances. 

\subsubsection{\sy{Inclusion and Decomposition Functions}}
\begin{defn}[Inclusion Functions]\cite[Chapter 2.4]{jaulinapplied} \label{def:inc_func}
Consider a function $f:\mathcal{Z} \subset \mathbb{R}^{n_z} \to \mathbb{R}^{n_x}$. The interval function $T^f:\mathbb{IR}^{n_z} \to \mathbb{IR}^{n_x}$ is an inclusion function for $f(\cdot)$, if 
\begin{align*}
\forall \sy{\mathbb{I}}\mathcal{Z} \in \mathbb{IR}^{n_z}, f(\sy{\mathbb{I}}\mathcal{Z}) \subset T^f(\sy{\mathbb{I}}\mathcal{Z}), 
\end{align*}
where $f(\sy{\mathbb{I}}\mathcal{Z})$ is the true image set \sy{(or range)} of $f(\cdot)$ 
\sy{for the domain} $\sy{\mathbb{I}}\mathcal{Z} \in \mathbb{IR}^{n_z}$. \sy{The tightest enclosing interval of $f(\sy{\mathbb{I}}\mathcal{Z})$ is denoted by $T^f_O(\mathbb{I}\mathcal{Z}) \triangleq [\underline{f}^{\text{true}} ,\overline{f}^{\text{true}} ] \mo{\triangleq [ \min\limits_{z \in \mathbb{I}\mathcal{Z}} f(z),  \max\limits_{z \in \mathbb{I}\mathcal{Z}} f(z)]} \supset f(\sy{\mathbb{I}}\mathcal{Z})$; hence, it is the tightest inclusion function, i.e., $T^f_O(\mathbb{I}\mathcal{Z}) \subseteq {{T}}^f(\sy{\mathbb{I}}\mathcal{Z})$. Further, with a slight abuse of notation, we overload the notation of $f(\mathbb{I}\mathcal{Z})$ and $T^f(\sy{\mathbb{I}}\mathcal{Z})$ 
when \eqref{eq:mix_mon_def} is a continuous-time system to represent $f_c(\mathbb{I}\mathcal{Z})\subset \mathbb{R}^{n_x}$ and $T^{f_c}(\sy{\mathbb{I}}\mathcal{Z}) \in \mathbb{I}\mathbb{R}^{n_x}$, respectively, with $ f_{c,i}(\mathbb{I}\mathcal{Z})\triangleq f_{i}(\mathbb{I}\mathcal{Z}_{c,i})$, $ T^{f_c}_i(\mathbb{I}\mathcal{Z})\triangleq T^{f}_i(\mathbb{I}\mathcal{Z}_{c,i})$ and $\mathbb{I}\mathcal{Z}_{c,i}\triangleq \{\syo{[\underline{z},\overline{z}] 
\in \mathbb{I}\mathcal{Z}} \, | \, [\underline{z}_j,\overline{z}_j] =\mathbb{I}\mathcal{Z}_j, \forall j \neq i, \underline{z}_i=\overline{z}_i=x_i\}$, \syo{$\forall i \in \mathbb{N}_{n_x}$}.}
\end{defn}
\begin{prop}[Natural ($T^f_N$) Inclusion Functions] \cite[Theorem 2.2]{jaulinapplied}\label{prop:natural}
Consider $\sy{\mathbb{I}}\mathcal{Z} \triangleq [\underline{z},\overline{z}] \in \mathbb{IR}^{n_z}$ and $f\triangleq [f_1,\dots, f_{n_x}]^\top:
\sy{\mathcal{Z} \subset \mathbb{R}^{n_z}} \to \mathbb{R}^{n_x}$, where each $f_j, \ j\in \sy{\mathbb{N}_{n_x}}$, \syo{is} 
expressed as a finite composition of the operators $+,-,\times,/$ and elementary functions (sine, cosine, exponential, square root, \sy{etc.).} 
A natural inclusion function $T^f_N:\mathbb{IR}^{n_z} \to \sy{\mathbb{I}}\mathbb{R}^{n_x}$ for $f(\cdot)$ is obtained by replacing each real variable $z_i,i\in \sy{\mathbb{N}_{n_z}}$, 
by its corresponding \emph{interval variable} $[z_i] \triangleq \sy{\mathbb{I}}\mathcal{Z}_i=[\underline{z}_i,\overline{z}_i]$, and each operator or function by its interval counterpart 
by applying \emph{interval arithmetic} (cf. \cite[Chapter 2]{jaulinapplied} \mo{for details}). 
\end{prop}
\begin{prop}[Centered ($T^f_C$) and Mixed Centered ($T^f_M$) Inclusion Functions]\cite[Sections 2.4.3--2.4.4]{jaulinapplied}\label{prop:natural_inclusions}
Let $f\triangleq [f_1,\dots, f_{n_x}]^\top:
\sy{\mathcal{Z} \subset \mathbb{R}^{n_z}} \to \mathbb{R}^{n_x}$ be differentiable over the 
\sy{interval} $\sy{\mathbb{I}}\mathcal{Z} \triangleq [\underline{z},\overline{z}] \in \mathbb{IR}^{n_z}$. Then, the interval function 
\begin{align*}
T^f_C(\sy{\mathbb{I}}\mathcal{Z})\triangleq f(m)+
\sy{\mathbb{I}\mathbf{J}^f_{\mathbb{I}\mathcal{Z}}}(\sy{\mathbb{I}}\mathcal{Z}-m),
\end{align*}
is an inclusion function for $f(\cdot)$, \sy{called the \emph{centered inclusion function} for $f(\cdot)$ in $\sy{\mathbb{I}}\mathcal{Z}$}, where $m \triangleq \frac{\underline{z}+\overline{z}}{2}$, 
$\sy{\mathbb{I}\mathbf{J}^f_{\mathbb{I}\mathcal{Z}}}$ is an \emph{interval \sy{Jacobian} matrix} \sy{with domain $\mathbb{I}\mathcal{Z}$}, such that $J_f(z) \in \sy{\mathbb{I}\mathbf{J}^f_{\mathbb{I}\mathcal{Z}}}, \forall z \in \sy{\mathbb{I}}\mathcal{Z}$ and $J_f(z)$ is the Jacobian matrix of $f(z)$ at point $z \in \sy{\mathbb{I}}\mathcal{Z}$. 
Moreover,
\begin{align*}
&T^f_M(\sy{\mathbb{I}}\mathcal{Z}) \triangleq [T^f_{M,i}(\sy{\mathbb{I}}\mathcal{Z}) \dots T^f_{M,n_x}(\sy{\mathbb{I}}\mathcal{Z})]^\top,
\end{align*}
{with} 
$T^f_{M,i}(\sy{\mathbb{I}}\mathcal{Z})\triangleq f_i(m)+\sum_{j=1}^{n_z} \sy{(\mathbb{I}}\mathbf{J}^{f}_{\sy{\mathbb{I}}\mathcal{Z}_{1 \to j}})_{i,j}(\sy{\mathbb{I}\mathcal{Z}_j}-m_j)$ for all $ i \in \sy{\mathbb{N}_{n_x}}$, 
 is also an inclusion function for $f(\cdot)$, called the \emph{mixed-centered inclusion function}, where $\sy{\mathbb{I}}\mathcal{Z}_j \triangleq [\underline{z}_j \ \overline{z}_j ], j\in \sy{\mathbb{N}_{n_z}}$, 
 $\sy{\mathbb{I}}\mathcal{Z}_{1 \to j} \triangleq [\sy{\mathbb{I}}\mathcal{Z}_1 \dots \sy{\mathbb{I}}\mathcal{Z}_j \ m_{j+1} \dots m_{n_z}]^\top$, $(\sy{\mathbb{I}}\mathbf{J}^{f}_{\sy{\mathbb{I}}\mathcal{Z}_{1 \to j}})_{i,j}$ is the $(i,j)$-th element of the interval \sy{Jacobian} matrix $\sy{\mathbb{I}}\mathbf{J}^{f}_{\sy{\mathbb{I}}\mathcal{Z}_{1 \to j}}$ \sy{with domain $\mathbb{I}\mathcal{Z}_{1 \to j}$}, and $\underline{z}_j,\overline{z}_j,m_j$ \mo{are} the $j$-th elements of the vectors $\underline{z},\overline{z},m$, respectively.   
\end{prop}
Next, inspired by the work in \cite[Section 3]{cornelius1984computing}, we introduce the notion of remainder-form (additive) inclusion functions.
\begin{defn}[Remainder-Form (Additive) Inclusion Functions] \label{defn:remainder_dec}
Consider a function $f:\sy{\mathcal{Z}\in\mathbb{R}^{n_z}} \to \mathbb{R}^{n_x}$. The interval function $T^f_R:\mathbb{IR}^{n_z} \to \mathbb{IR}^{n_x}$ is an additive (remainder-form) inclusion function for $f(\cdot)$, if there exist two constituent mappings $g,h:\mathbb{R}^{n_z} \to \mathbb{R}^{n_x}$, such that for any $\sy{\mathbb{I}}\mathcal{Z} \in \mathbb{IR}^{n_z}$:
\begin{align*} 
f(\sy{\mathbb{I}}\mathcal{Z})\subseteq T^f_R(\sy{\mathbb{I}}\mathcal{Z}) \triangleq g(\sy{\mathbb{I}}\mathcal{Z})+h(\sy{\mathbb{I}}\mathcal{Z}),
\end{align*}
where addition is based on interval arithmetic \sy{(cf. \cite{jaulinapplied})}.
\end{defn}
\begin{defn}[Mixed-Monotonicity \sy{and} Decomposition Functions] \cite[Definition 1]{abate2020tight},\cite[Definition 4]{yang2019sufficient} \label{defn:dec_func}
\syo{Consider the dynamical system \sy{with initial state $x_0 \in \mathbb{I}\mathcal{X}_0 \triangleq [\underline{x}_0,\overline{x}_0]$:}
\begin{align}\label{eq:mix_mon_def}
x_t^+=\tilde{f}(x_t,u_t,w_t)\triangleq f(z_t),
\end{align}
where $x_t^+ \triangleq x_{t+1}$ if \eqref{eq:mix_mon_def} is a discrete-time \sy{system} and $x_t^+ \triangleq \dot{x}_t$ if \eqref{eq:mix_mon_def} is a continuous-time system, $\tilde{f}:\mathcal{X}\times \mathcal{U}\times \mathcal{W}\to \mathbb{R}^{n_x}$ is the vector field 
with state $x_t \in \mathcal{X} 
\subset \mathbb{R}^{n_x}$, known input $u_t \in \mathcal{U} \subset \mathbb{R}^{n_u}$ and disturbance input $w_t \in \mathcal{W} \subseteq \mathbb{I}\mathcal{W} \triangleq [\underline{w},\overline{w}] \in \mathbb{IR}^{n_w}$. For ease of exposition, we also define $f : \mathcal{Z} \triangleq \mathcal{X} \times \mathcal{W} 
\subset \mathbb{R}^{n_z} \to \mathbb{R}^{n_x}$ as in \eqref{eq:mix_mon_def} that is implicitly dependent on $u_t$ with the augmented state $z_t\triangleq [x_t^\top  \ w_t^\top]^\top \in \mathcal{Z}$.}

Suppose \eqref{eq:mix_mon_def} is a discrete-time system. Then, a mapping $f_d:\mathcal{Z}\times \mathcal{Z} \to \mathbb{R}^{n_x}$ is 
a \sy{discrete-time  mixed-monotone} decomposition function with respect to $f(\cdot)$, 
if it satisfies: 
\begin{enumerate}[i)]
\item $f(\cdot)$ is embedded on the diagonal of $f_d(\cdot,\cdot)$, i.e., $f_d(z,z)=f(z)$. \label{item:1}
\item $f_d$ is monotone increasing in its first 
argument, i.e., $\hat{z}\ge z \implies f_d(\hat{z},z') \geq f_d(z,z')$. 
\label{item:2}
\item \sy{$f_d$ is monotone decreasing in its second argument, i.e., $\hat{z}\ge z \implies f_d(z',\hat{z}) \leq f_d(z',z).$} \label{item:3}
\end{enumerate} 
\sy{Further, if \eqref{eq:mix_mon_def} is a continuous-time system, 
a mapping $f_d:\mathcal{Z}\times \mathcal{Z} \to \mathbb{R}^{n_x}$ is a \sy{continuous-time  mixed-monotone} decomposition function with respect to $f(\cdot)$, 
if it satisfies: 
\begin{enumerate}[i)]
\item $f(\cdot)$ is embedded on the diagonal of $f_d(\cdot,\cdot)$, i.e., $f_d(z,z)=f(z)$. \label{item:1}
\item $f_d$ is monotone increasing in its first 
argument with respect to ``off-diagonal'' arguments, i.e., $\forall i \in \mathbb{N}_{n_x}$, $\hat{z}_j\ge z_j, \forall j \in \mathbb{N}_{n_z}, \hat{z}_i= z_i=x_i   \implies f_{d,i}(\hat{z},z') \geq f_{d,i}(z,z')$.
\label{item:2}
\item \sy{$f_d$ is monotone decreasing in its second argument, i.e., $\hat{z}\ge z \implies f_d(z',\hat{z}) \leq f_d(z',z).$} \label{item:3}
\end{enumerate} 
Moreover, systems that admit mixed-monotone decomposition functions are called mixed-monotone systems.}
%
\end{defn}

Moreover, we extend the concept of decomposition functions \syo{to} \emph{one-sided decomposition functions}.
\begin{defn}[One-Sided Decomposition Functions] \label{def:one_side_dec}
Consider $f:\sy{\mathcal{Z}\subset\mathbb{R}^{n_z}} 
\to \mathbb{R}^{n_x}$ and suppose 
there exist two   mixed-monotone mappings $\overline{f}_d,\underline{f}_d:\mathcal{Z}\times \mathcal{Z} \to \mathbb{R}^{n_x}$ such that for any \sy{$\underline{z}, z, \overline{z} \in \mathcal{Z}$, the following statement holds:
\begin{align}\label{eq:one-sided}
\underline{z} \leq z \leq \overline{z} \implies \underline{f}_d(\underline{z},\overline{z}) \leq f(z) \leq \overline{f}_d(\overline{z},\underline{z}).
\end{align}
Then, $\overline{f}_d$ and $\underline{f}_d$ are called upper and lower decomposition functions for $f$ over $\mathbb{I}\mathcal{Z}$}, respectively.
\end{defn}

\sy{Based on the above definitions, the mixed monotone decomposition function can be viewed as the special case when the upper and lower decomposition functions coincide. Moreover, one-sided decomposition functions can be obtained from a family of (upper and lower) decomposition functions that are generally not unique via the following result:
\begin{cor}[Intersection Property]\label{cor:min_max_dec}
Suppose $\overline{f}^1_d,\overline{f}^2_d$ and $\underline{f}^1_d,\underline{f}^2_d$ are pairs of upper and lower decomposition functions for $f$, respectively. Then, $\min\{\overline{f}^1_d,\overline{f}^2_d\}$ and $\max\{\underline{f}^1_d,\underline{f}^2_d\}$ (i.e., their \emph{intersection}) are also upper and lower decomposition function for $f(\cdot)$, respectively. 
\end{cor}
\begin{proof}
The results follow from the fact that if two mappings with the same domains and image spaces are both monotonically increasing or decreasing in their $i$-th arguments, then their ``point-wise" minimum and maximum are also \mo{monotonically} increasing or decreasing on their $i$-th arguments.  
\end{proof}}

Further, we can slightly generalize the notion of \emph{embedding system with respect to $f_d$} in \cite[(7)]{abate2020tight}, to the \emph{embedding system with respect to $\overline{f}_d,\underline{f}_d$}, through the following definition.
\begin{defn}[Generalized Embedding \sy{Systems}]\label{def:embedding}
For an $n$-dimensional 
system \eqref{eq:mix_mon_def} \sy{with any one-sided decomposition functions $\overline{f}_d,\underline{f}_d$, its   \emph{generalized embedding system is the} 
$2n$-dimensional system \sy{with initial condition $\begin{bmatrix} \overline{x}_0^\top & \underline{x}_0^\top\end{bmatrix}^\top$}:}
\begin{align} \label{eq:embedding}
\syo{\begin{bmatrix}{\overline{x}}_t^+ \\ {\underline{x}}_t^+ \end{bmatrix}=\begin{bmatrix} \overline{f}_d(\begin{bmatrix}(\overline{x}_t)^\top \, \overline{w}^\top\end{bmatrix}^\top\hspace{-0.05cm},\begin{bmatrix}(\underline{x}_t)^\top \, \underline{w}^\top \end{bmatrix}^\top) \\ \underline{f}_d(\begin{bmatrix}(\underline{x}_t)^\top \, \underline{w}^\top \end{bmatrix}^\top\hspace{-0.05cm},\begin{bmatrix}(\overline{x}_t)^\top \, \overline{w}^\top\end{bmatrix}^\top) \end{bmatrix}.} 
\end{align}
\end{defn}

\begin{prop}[State Framer Property]\label{cor:embedding} 
\sy{Let system \eqref{eq:mix_mon_def} with initial state $x_0 \in \mathbb{I}\mathcal{X}_0 \triangleq  [\underline{x}_0,\overline{x}_0]$ be mixed-monotone with a 
generalized embedding system \eqref{eq:embedding} with respect to 
$\overline{f}_d,\underline{f}_d$. Then, for all $t\ge0$, $R^f(t,\mathbb{I}\mathcal{X}_0) \subset \mathbb{I}\mathcal{X}_t \triangleq [\underline{x}_t,\overline{x}_t]$, 
where $R^f(t,\mathbb{I}\mathcal{X}_0) \triangleq
\{\phi(t, x_0, {w}\syo{_{0:t}}) \mid x_0 \in \mathbb{I}\mathcal{X}_0 \text{ and } w_t \in \mathbb{I}\mathcal{W},  \forall t\ge 0\}$ is the reachable set at time $t$ of \eqref{eq:mix_mon_def} 
when initialized within $\mathbb{I}\mathcal{X}_0$ and $ \mathbb{I}\mathcal{X}_t \triangleq [\underline{x}_t,\overline{x}_t]$ is the solution to the generalized embedding system \eqref{eq:embedding}. Consequently,
the system state trajectory $x_t$ 
satisfies $\underline{x}_t \le x_t \le \overline{x}_t, \mo{\forall t \geq 0}$, i.e., is \emph{framed} by $ \mathbb{I}\mathcal{X}_t$.}
\end{prop}
\begin{proof}
\sy{The proof for continuous-time systems is similar to the proof in \cite[Proposition 3]{cooganmixed},}
\sy{while the discrete-time result follows from repeatedly applying its definition in \eqref{eq:one-sided}.}
\end{proof}

\sy{Next, the following observation, which follows directly from Definitions \ref{def:inc_func}, \ref{defn:dec_func} and \ref{def:one_side_dec}, relates the concept of decomposition functions to inclusion functions.}

\begin{prop}[Decomposition-Based Inclusion Functions] \label{cor:dec_inc}
\sy{Given  any upper and lower decomposition functions  $\overline{f}_d,\underline{f}_d$ (or any decomposition function $f_d=\overline{f}_d=\underline{f}_d$) for $f$, 
$$T^{f_d} 
(\mathbb{I}\mathcal{Z}) \triangleq [\underline{f}_d(\underline{z},\overline{z}),\overline{f}_d(\overline{z},\underline{z})]$$
satisfies 
$f(\mathbb{I}\mathcal{Z}) \subset T^{f_d} 
(\mathbb{I}\mathcal{Z})$ with $\mathbb{I}\mathcal{Z} \triangleq [\underline{z},\overline{z}]$ (including overloading; cf. Definition \ref{def:inc_func}). 
Consequently, $T^{f_d} 
(\mathbb{I}\mathcal{Z})$ is an inclusion function (that is based on 
decomposition functions).}
\end{prop}

\sy{As noted earlier, (mixed-monotone) decomposition functions defined in Definition \ref{defn:dec_func} are not unique. Hence, a measure of their tightness is beneficial for comparing these functions.}

\begin{defn}[Tightness of Decompositions] \cite[Definition 2]{abate2020tight}\label{defn:tightness}
A decomposition function $f^1_d$ for system \eqref{eq:mix_mon_def} is tighter than decomposition function $f^2_d$, if \sy{for all} $z \leq \hat{z}$,
\begin{align} \label{eq:tightness_dec}
f^2_d(z,\hat{z}) \leq f^1_d(z,\hat{z}) \ \text{and} \ f^1_d(\hat{z},{z}) \leq f^2_d(\hat{z},{z}). 
\end{align} 
Then, 
\sy{$f^O_d$} is tight, i.e., it is the tightest possible decomposition function for $f$, if \eqref{eq:tightness_dec} holds with $f^1_d=\sy{f^O_d}$ and any other decomposition function $f^2_d$. 

\sy{Furthermore, we define the \emph{metric/measure} of tightness as the maximum dimension-wise Hausdorff distance given by:
\begin{align} \label{eq:metric}
q(f(\mathbb{I}\mathcal{Z}),T^{f_d} 
(\mathbb{I}\mathcal{Z}))\triangleq \max_{i\in \mathbb{N}_{n_x}}\tilde{q}(f_i(\mathbb{I}\mathcal{Z}),T^{f_{d_i}} 
(\mathbb{I}\mathcal{Z})),
\end{align}
where $\tilde{q}(\sy{\mathbb{I}}\mathcal{X}_1,\sy{\mathbb{I}}\mathcal{X}_2)\triangleq \max \{|\underline{x}_1-\underline{x}_2|_{\infty},|\overline{x}_1-\overline{x}_2|_{\infty}\}$ is the Hausdorff distance between 
two real intervals $\sy{\mathbb{I}}\mathcal{X}_1=[\underline{x}_1,\overline{x}_1]$ and $\sy{\mathbb{I}}\mathcal{X}_2=[\underline{x}_2,\overline{x}_2]$, both in $\mathbb{IR}$ \cite{cornelius1984computing}.  
Moreover, the above tightest decomposition function $T_O^{f_d}(\mathbb{I}\mathcal{Z}) \triangleq [f_d^O(\underline{z},\overline{z}),f_d^O(\overline{z},\underline{z})]$ satisfies $q(f(\mathbb{I}\mathcal{Z}),T_O^{f_d}(\mathbb{I}\mathcal{Z}))=0$.}
\end{defn}


\syo{Further, by Proposition \ref{cor:dec_inc}, we can obtain 
inclusion functions from existing 
decomposition functions as shown in Propositions \ref{prop:Liren_dec} and \ref{prop:tight_dec}. 
Note that we slightly modified the decomposition function in \cite[Theorem 2]{yang2019sufficient} to make it also applicable for continuous-time systems and for closed Jacobian intervals. }

\begin{prop}[$T^{f_d}_L$ Inclusion Functions]
\label{prop:Liren_dec}
\sy{For any system in the form of \eqref{eq:mix_mon_def}, suppose that $f:\mathbb{R}^{n_z} \to \mathbb{R}^{n_x}$ is differentiable and $\frac{\partial f_i}{\partial z_j}(z)\in [a_{ij},b_{ij}], \forall z \in \sy{\mathbb{I}}\mathcal{Z}\sy{\triangleq [\underline{z},\overline{z}]}\subseteq \sy{\mathbb{I}}\mathbb{R}^{n_z}$. Then, 
a \sy{discrete-time or continuous-time}  mixed-monotone decomposition function $f^L_{d}=[f^L_{d,1}\dots f^L_{d,n_x}]$ with respect to $f(\cdot)$ and its corresponding inclusion function $T^{f_d}_L(\mathbb{I}\mathcal{Z})=[f^L_d(\mo{\underline{z}},\mo{\overline{z}}),f^L_d(\mo{\overline{z}},\mo{\underline{z}})]$ can be 
\sy{constructed} as follows:}

\sy{For all $ i \in \mathbb{N}_{n_x}$ and $j \in \mathbb{N}_{n_z}$, 
\begin{align}\label{eq:Lir_dec}
f^L_{d,i}({z},\hat{z})=f_i({\zeta})+(\alpha_i-\beta_i)({z}-\hat{z}),
\end{align}
with  $\alpha_i=[\alpha_{i1},\dots,\alpha_{in_z}]$, $\beta_i=[\beta_{i1},\dots,\beta_{in_z}]$ and $\zeta=[\zeta_1,\dots,\zeta_{n_z}]^\top$, where\\
 $\alpha_{ij}=\begin{cases} 0, &\hspace{-0.3cm} \text{Cases} \, 1,3,4,5, \\ |a_{ij}|, & \hspace{-0.3cm}\text{Case} \, 2,\end{cases}$\hspace{-0.1cm} $\beta_{ij}=\begin{cases} 0, &\hspace{-0.3cm} \text{Cases} \,  1,2,4,5, \\ -|b_{ij}|, &\hspace{-0.3cm} \text{Case} \,  3,\end{cases}$ 
and $\zeta_j=\begin{cases} \underline{z}_j, \ \text{Cases} \ 1,2,  \\ \overline{z}_j, \ \text{Cases} \ 3,4,\\ x_j, \ \text{Case} \ 5,\end{cases}$ with Cases 1 through 4 for discrete-time systems and Cases 1 through 5 for continuous-time systems, 
\syo{where the cases as defined as:} Case $1: a_{ij}\geq 0$, Case $2: a_{ij}\leq 0,b_{ji}\geq 0, |a_{ij}|\leq |b_{ij}|$, Case $3: a_{ij}\leq 0,b_{ij}\geq 0, |a_{ij}|\geq |b_{ij}| $, Case $4: b_{ij}\leq 0$ 
and Case 5: $j=i$.} 
%
\end{prop}

\begin{prop}[Tight Decomposition Functions for Mixed-Monotone Systems]\cite[Theorem 2]{yang2019tight},\cite[Theorem 1]{abate2020tight} \label{prop:tight_dec}
For any system in the form of \eqref{eq:mix_mon_def} \mo{and $\mathbb{I}\mathcal{Z} \triangleq [\underline{z},\overline{z}]$}, 
a tight (optimal) \sy{discrete-time or continuous-time}  mixed-monotone decomposition function $f^O_d=[f^O_{d,1} \dots f^O_{d,n_x}]$ \sy{and its corresponding tight inclusion function $T^{f_d}_O(\mathbb{I}\mathcal{Z})\triangleq[f^O_d(\underline{z},\overline{z}),f^O_d(\overline{z},\underline{z})]$ (i.e., $T^f_O(\mathbb{I}\mathcal{Z})=T^{f_d}_O(\mathbb{I}\mathcal{Z})$)} can be \sy{constructed as follows:}  

If \eqref{eq:mix_mon_def} is a discrete-time \syo{system}, then $\forall i\in \mathbb{N}_{n_x}$:
\begin{align}\label{eq:tight_dec_disc}
f^O_{d,i}(z,\hat{z})=\begin{cases} \min\limits_{\zeta \in [z,\hat{z}]} f_i(\zeta) \quad \text{if} \ z\leq\hat{z}, \\
\max\limits_{\zeta \in [\hat{z},{z}]} f_i(\zeta) \quad \text{if} \ \hat{z}\leq{z}. \end{cases}
\end{align}
Moreover, if \eqref{eq:mix_mon_def} is a continuous-time system, then $\forall i \in \mathbb{N}_{n_x}$:
\begin{align}\label{eq:tight_dec}
f^O_{d,i}(z,\hat{z})=\begin{cases} \min\limits_{\zeta \in [z,\hat{z}],\zeta_i=x_i} f_i(\zeta) \quad \text{if} \ z\leq\hat{z}, \\[-0.1cm]
\max\limits_{\zeta \in [\hat{z},{z}],\zeta_i=x_i} f_i(\zeta) \quad \text{if} \ \hat{z}\leq{z}. \end{cases}
\end{align}
\end{prop}



\sy{
\begin{rem}
On the flip side, the discrete-time mixed-monotone decomposition functions can also be directly used as inclusion functions; hence, the proposed decomposition functions can also be relatively easily incorporated into existing analysis, estimation and planning algorithms that are based on \emph{interval arithmetic}, e.g., \cite{jaulinapplied,cornelius1984computing,moore2009introduction,alefeld2000interval,kieffer2002guaranteed}. 
\end{rem}
}

Although Proposition \ref{prop:tight_dec} provides theoretically tight decomposition functions, 
it has some limitations in practice (see \cite[Section III]{cooganmixed} for a detailed discussion). For instance, exact closed-form solutions to the nonlinear programs in  \eqref{eq:tight_dec_disc} and  \eqref{eq:tight_dec} may not always be available. \sy{With this in mind, we define the notion of 
 tractability of decomposition functions as follows.}
\begin{defn}[Tractable Decomposition Functions]\label{defn:tractability}
$\overline{f}_d,\underline{f}_d$ are \sy{computationally} tractable/computable one-sided decomposition functions for mapping $f$, if they can be \sy{constructed in closed-form\mo{,} i.e., with a finite number of elementary operations and without differentiation nor an iterative procedure.} 
 \end{defn}
 
 \sy{Thus, a corollary of Proposition \ref{prop:tight_dec} is as follows.}
 \begin{cor}[Tight \sy{and Tractable} Decomposition \sy{Function}s for JSS Vector Fields] \label{cor:jss_tight}
\moh{Suppose} $f(\cdot)$ is continuously differentiable and Jacobian sign-stable (JSS) \cite{yang2019sufficient}, i.e., 
\sy{$\forall i \in \mathbb{N}_{n_x}, \forall j \in \mathbb{N}_{n_z}, J^f_{ij}(z)\triangleq \frac{\partial f_i}{\partial z_j}(z)\geq 0, \forall z \in  \mathbb{I}\mathcal{Z}$ or $J^f_{ij}(z)\triangleq \frac{\partial f_i}{\partial z_j}(z)\leq 0, \forall z \in  \mathbb{I}\mathcal{Z}$.} Then, the following statements hold\syo{:}
\begin{enumerate}[i)]
\moh{\item $\forall i \in \mathbb{N}_{n_x}, \forall j \in \mathbb{N}_{n_z}$, $f_i(\cdot)$ in either monotonically non-decreasing or monotonically non-increasing in its $j$-th argument $z_j$, over the entire domain $\mathbb{I}\mathcal{Z}$\syo{;}}  
\item The optimization programs in \eqref{eq:tight_dec_disc} and \eqref{eq:tight_dec} can be \sy{tractably} and 
exactly solved by enumerating $f_i(\cdot)$ \sy{at} the vertices of \sy{$\mathbb{I}\mathcal{Z}$  (with fixed $\zeta_i=x_i$ for continuous-time systems) and choosing the corresponding optima.} 
 \end{enumerate} 
\end{cor}

 \subsubsection{\sy{A Novel Class of Mixed-Monotone Systems}}
 
\sy{To describe our modeling framework, we 
formally define a novel class of nonsmooth 
systems, which we show to include a 
wide range of nonlinearities 
and prove  in Section \ref{sec:main} to be mixed-monotone.} 
\begin{defn}[Either-Sided Locally Lipschitz Systems] \label{defn:locally-Lip}
System \eqref{eq:mix_mon_def} is \emph{either-sided locally Lipschitz continuous} (ELLC) 
\sy{\mo{at} $z \in \mathcal{Z}$,} \mo{if} there exists an open neighborhood $\mathcal{N}_{z} \subset \mathcal{Z}$ of $z$\mo{,} \sy{and for all $i \in \mathbb{N}_{n_x}$, there exist vectors $\kappa_i \triangleq [\kappa_{i1} \dots \kappa_{in_z}] \in \mathbb{R}^{n_z}$ with non-zero elements (i.e., $\kappa_{ij}\ne 0, \forall j \in \mathbb{N}_{n_z}$) 
and constants $\rho_i \in \mathbb{R}$} 
such that $\forall z',z'' \in \mathcal{N}_{z}$, 
\begin{align}\label{eq:one-sided-Lip}
\sy{\moh{\langle}\kappa_i(f_i(z')\hspace{-0.07cm}-\hspace{-0.07cm}f_i(z'')), z' \hspace{-0.07cm}-\hspace{-0.07cm}z''\moh{\rangle} \hspace{-0.07cm} \leq \rho_i\|z'\hspace{-0.07cm}-\hspace{-0.07cm}z''\|^2_2, \ \forall i \in \mathbb{N}_{n_x},}
\end{align}
where 
\sy{$\langle \cdot,\cdot\rangle$} denotes the inner product operator. \sy{Further, if \eqref{eq:one-sided-Lip} holds $\forall i \in \mathbb{N}_{n_x}$, we also call $f$ an ELLC function on $\mathcal{Z}$.}
\end{defn}
Note that \eqref{eq:one-sided-Lip} holds for a very broad range of nonlinear systems. Particularly, \sy{ELLC systems reduce to one-sided locally Lipschitz systems}
\footnote{System \eqref{eq:mix_mon_def} with $n_x=n_z$ is one-sided locally Lipschitz continuous 
\sy{if for all $z \in \mathcal{Z}$, there exist an open neighborhood $\mathcal{N}_{z} \subset \mathcal{Z}$ of $z$ and a} ${\rho} \in \mathbb{R}$ such that $(f(z')-f(z''))^\top (z'-z'') \leq {\rho}\|z'-z''\|^2_2, \forall z',z'' \in \mathcal{N}_z$.}, 
\sy{when $n_z=n_x$ and $\kappa_{ii}=\kappa_{jj} \triangleq \kappa > 0$ and $\kappa_{ij}=0$, $\forall i \neq j$,
with 
Lipschitz constant $\frac{1}{\kappa}\sum_{i=1}^{n_x}\rho_i$. 
Moreover, the class of ELLC systems is a superset of the class of locally Lipschitz continuous (LLC) systems\footnote{System \eqref{eq:mix_mon_def} is 
\sy{LLC if for all $z \in \mathcal{Z}$, there exist an open neighborhood $\mathcal{N}_{z}$ 
of $z$ and a ${\rho} \ge 0$ 
such that $\| f(z')-f(z'')\|_2 \le \rho \|z'-z''\|_2$, 
$\forall z',z''\in \mathcal{N}_z$.}}, 
as shown next.}
\begin{prop}\label{prop:Lip_ELLC}
An LLC system is also an ELLC system. 
\end{prop} 
\begin{proof}\sy{Suppose a system is LLC, then each $f_i$ is also LLC with some $\tilde{\rho}_i$. Next, by applying the Cauchy-Schwartz inequality and with each $f_i$ being LLC and any $\kappa_i$ with non-zero elements, we have}
$\moh{\langle}\kappa_i(f_i(z')-f_i(z)), z'-z\moh{\rangle} \leq |\kappa_i||f_i(z')-f_i(z)|\|z'-z\|_2 \leq \rho_i \|z'-z\|^2_2$, where $\rho_i \triangleq |\kappa_i| \tilde{\rho}_i$. 
\end{proof} 
Next, we 
\sy{review a} notion of generalized gradients 
used in \emph{nonsmooth} analysis in systems and control theory \cite{clarke2009nonsmooth} when \sy{vector fields of the systems} are not necessarily differentiable.   
\begin{defn}[Clarke Generalized \syo{D}irectional \syo{D}erivatives]\cite[Chapter II]{demianov1995constructive}\label{defn:gen_grad}
\sy{Given a function} $f :\mathcal{Z} \subseteq \mathbb{R}^{n_z} \to \mathbb{R}$, 
\begin{align}\label{eq:key-clarke-bound}
\begin{array}{ll}
 {f_{C}^\uparrow}(z,v) \triangleq \limsup\limits_{t\to z,\lambda \downarrow 0}\frac{f(t+\lambda v)-f(t)}{\lambda}=\sup\limits_{\xi \in {\partial}_C f(z)} \xi^\top v, \\
 {f_{C}^\downarrow}(z,v) \triangleq \liminf\limits_{t\to z,\lambda \downarrow 0}\frac{f(t+\lambda v)-f(t)}{\lambda}=\inf\limits_{\xi \in {\partial}_C f(z)} \xi^\top v, 
 \end{array}
\end{align}
are 
the (generalized) Clarke upper and lower directional derivatives\sy{/gradients} of $f(\cdot)$ at $z \in \mathcal{Z}$ in the direction $v \in \mathbb{R}^{n_z}$, respectively, where 
 the set
\begin{align*}
{\partial}_C f(z) \triangleq \{\xi \in \mathbb{R}^{n_z} | {f_{C}^\downarrow}(z,v) \leq \xi^\top v \leq {f_{C}^\uparrow}(z,v), \forall v \in \mathbb{R}^{n_z} \},
\end{align*}
 is 
 the Clarke sub-differential (set) of $f(\cdot)$ at $z\in \mathcal{Z}$. 
\end{defn}
\sy{Note that, by definition,} \sy{$ {f_{C}^\uparrow}(z,v)$ and $ {f_{C}^\downarrow}(z,v)$} are the \emph{upper} and \emph{lower support functions} of the set ${\partial}_C f(z)$. 
\sy{Further,} as shown in \cite[\mo{Appendix I}]{demianov1995constructive}, 
$\partial_C f(z)$ \sy{for an LLC system} is nonempty, convex and compact 
for all $z \in \mathcal{Z}$, and consequently, at each 
$z \in \mathcal{Z}$, \sy{the Clarke directional derivatives are bounded in each direction $v$.} 
However, this does not hold in general for ELLC systems. 
\sy{Nonetheless, \mo{by} the following proposition\mo{,}} 
the Clarke directional derivatives \sy{in some specific directions} are bounded from above \emph{or} from below. 
\begin{prop} \label{prop:ELLC} 
Suppose $f :\mathcal{Z} \subseteq \mathbb{R}^{n_z} \to \mathbb{R}$ is ELLC on $\mathcal{Z}$ and 
let \sy{$e_j, \forall j \in \mathbb{N}_{n_z}$} denote the standard unit vector\syo{/basis} in the $j$-th coordinate direction. 
Then, $f_C^\uparrow (z,e_j)$ is bounded from above or $f_C^\downarrow (z,e_j)$ is bounded from below. 
\end{prop}
\begin{proof}
Setting $z'=z+\lambda e_j$ in \eqref{eq:one-sided-Lip}, where $\lambda \in \mathbb{R}$ is sufficiently small, we obtain $\kappa_{ij}(f_i(z+\lambda e_j)-f_i(z))\lambda  \leq \rho_i\lambda^2$. Then, dividing both sides by $\kappa_{ij} \ne 0$ and taking the $\limsup$ if $\kappa_{ij} >0$, or the $\liminf$ if $\kappa_{ij} <0$, from both sides when $\lambda \to 0$, imply that $f_C^\uparrow (z,e_j)$ is bounded from above or $f_C^\downarrow (z,e_j)$ is bounded from below, respectively (cf. \eqref{eq:key-clarke-bound}).  
\end{proof}

\sy{Then, 
by Proposition \ref{prop:ELLC}, we \syo{can} easily show the following.} 
\begin{cor}\label{cor:Clrake-Jacob-def}
Consider the mapping $f=[f_1 \dots f_{n_x}]^\top:\mathcal{Z} \subset \mathbb{R}^{n_z} \to \mathbb{R}^{n_x}$, where $\forall i \in \mathbb{N}_{n_x}$, $f_i$ is ELLC on $\mathcal{Z}$. We define upper and lower Clarke Jacobian matrices of $f$ at $z \in \mathcal{Z}$, \sy{$\overline{J}^f_C(z)=[(\overline{J}^f_C(z))_{ij}]$ and $\underline{J}^f_C(z)=[(\underline{J}^f_C(z))_{ij}]$, with the upper and lower partial Clarke derivatives  \moh{at point $z \in \mathcal{Z}$} for each $i \in \sy{\mathbb{N}_{n_x}},j \in  \sy{\mathbb{N}_{n_z}}$, respectively, defined as: 
\begin{align}\label{Clarke-Jacob-matrix}
(\overline{J}^f_C(z))_{ij} \triangleq f^\uparrow_{i,C}(z,e_j), \ (\underline{J}^f_C(z))_{ij} \triangleq f^\downarrow_{i,C}(z,e_j).
\end{align}
Then,} 
$\forall i \in \mo{\mathbb{N}_{n_x}}, \forall j \in \mo{\mathbb{N}_{n_z}} $, 
$\overline{J}^f_C(z)_{ij}$ is bounded from above or/and $\underline{J}^f_C(z)_{ij}$ is bounded \syo{from} below, \mo{i.e.,} 
\sy{we cannot simultaneously have} $(\overline{J}^f_C(z))_{ij}=\infty$ and $(\underline{J}^f_C(z))_{ij}=-\infty$. 
\end{cor}
Next, we \sy{provide a slight} modification of the results in \cite[Proposition 1.12]{demianov1995constructive}, which plays an important role in our main results later and its proof goes precisely \syo{along} 
the lines of the proof of \cite[Proposition 1.12]{demianov1995constructive}.
\begin{prop}\label{prop:clarke-sum}
Let $f:\mathcal{Z} \subseteq \mathbb{R}^{n_z} \to \mathbb{R}^{n_x}$ \sy{be decomposable into} 
$f=f^1+f^2$, where $f^1,f^2:\mathcal{Z} \subseteq \mathbb{R}^{n_z} \to \mathbb{R}$ and $f_i,f^1_i,f^2_i$, $\forall i \in \sy{\mathbb{N}_{n_x}}$, are ELLC. 
Then, $\forall z \in \mathcal{Z}, \forall i \in \sy{\mathbb{N}_{n_x}}, \forall j \in \sy{\mathbb{N}_{n_z}}$: 

\phantom{ }\hspace{0.1cm}\vspace*{-0.9cm}
\begin{align*}
(\overline{J}^f_C(z))_{ij} &\leq (\overline{J}^{f^1}_C(z))_{ij}+(\overline{J}^{f^2}_C(z))_{ij}, \\ \relax
(\underline{J}^f_C(z))_{i,j} &\geq (\underline{J}^{f^1}_C(z))_{ij}+(\underline{J}^{f^2}_C(z))_{ij}.
\end{align*} 
\begin{proof}
The results follow from \eqref{Clarke-Jacob-matrix} and the facts that 
$f^\uparrow_C(z,v) \leq f^{1\uparrow}_C(z,v)+f^{2\uparrow}_C(z,v)$ \cite[Chapter II, Proposition 1.12]{demianov1995constructive} and $f^\downarrow_C(z,v) =-f^\uparrow_C(z,-v)$ \cite[Chapter II, Proposition 1.7]{demianov1995constructive}, $\forall z \in \mathcal{Z}, \forall v \in \mathbb{R}^{n_z}$.
\end{proof}
\end{prop}   
It is worth noticing that when $f(\cdot)$ is differentiable, then $\nabla f(x) \in \partial_C f(x)$ and if $f(\cdot)$ is continuously differentiable or strictly differentiable, then $\partial_C f(x) =\{\nabla f(x)\}$. 
Now, we are ready to explain the notion of \emph{Clarke Jacobian sign-stability} through the following definition, which is a generalization of Jacobian sign-stability in \cite{yang2019sufficient} 
(\sy{cf.} Corollary \ref{cor:jss_tight}). 
\begin{defn}[Clarke Jacobian Sign-Stability]\label{defn:Jac_sign_stab}
A mapping $f
:\mathcal{Z} \subset \mathbb{R}^{n_z} \to \mathbb{R}^{n_x}$ is called Clarke Jacobian sign-stable (CJSS) over $\mathcal{Z}$, if $\forall z \in \mathcal{Z}, \forall i \in \sy{\mathbb{N}_{n_x}}, \forall j \in \sy{\mathbb{N}_{n_z}}$, 
\begin{align} \label{eq:CJSS}
(\overline{J}^f_C(z))_{ij}\leq 0 \ \lor \  (\underline{J}^f_C(z))_{ij}\geq 0.
\end{align}
\end{defn}
Finally, 
\sy{we present} an extension of Corollary \ref{cor:jss_tight}, 
which we will apply later in our derivations.
\begin{prop}\label{prop:CJSS_optimality}
\sy{\moh{Suppose} $f$ is ELLC and CJSS over $\mathbb{I}\mathcal{Z}$. Then, $\forall i \in \mathbb{N}_{n_x}, \forall j \in \mathbb{N}_{n_z}$, $f_i(\cdot)$ in either monotonically non-decreasing or monotonically non-increasing in its $j$-th argument $z_j$, over the entire domain $\mathbb{I}\mathcal{Z}$ and consequently, the optima of \eqref{eq:tight_dec_disc} and \eqref{eq:tight_dec} are attained at some vertices of $\mathbb{I}\mathcal{Z}$.}
\end{prop} 
\begin{proof}
\sy{For $i \in \mathbb{N}_{n_x}$, consider any arbitrary $z^1$ \moh{from the interior of $\mathbb{I}\mathcal{Z}$ and construct} $z^2=z^1+\lambda e_j \in \mathbb{I}\mathcal{Z}$, for some small enough $\lambda>0$. Then, by \sy{an identifical proof to} \cite[Chapter II, Theorem 1.3]{demianov1995constructive}, 
there exists a $ z_{\theta} \in \mathbb{I}\mathcal{Z}$ on the connecting line between $z^1$ and $z^2$, such that $\lambda f^\downarrow_{i,C}(z_{\theta},e_j)=f^\downarrow_{i,C}(z_{\theta},\lambda e_j) \leq f_i(z^2)-f_i(z^1) \leq f^\uparrow_{i,C}(z_{\theta},\lambda e_j) =\lambda f^\uparrow_{i,C}(z_{\theta},e_j)$, where the equalities hold by \cite[Chapter II, Proposition 1.5]{demianov1995constructive}. Using this and the CJSS assumption (cf. \eqref{eq:CJSS}), 
$f_i(z^2) \leq f_i(z_2)$ if $f^\uparrow_{i,C}(z,\lambda e_j)=(\overline{J}^f_C(z))_{ij}\leq 0$ holds for all $z \in \mathbb{I}\mathcal{Z}$ \moh{(including $z=z_{\theta}$)}, 
and \moh{similarly,} $f_i(z^2) \geq f_i(z_2)$ if $f^\downarrow_{i,C}(z,\lambda e_j)=(\underline{J}^f_C(z))_{ij}\geq 0$ holds for all $z \in \mathbb{I}\mathcal{Z}$ \moh{(including $z=z_{\theta}$)}. 
Hence, by moving 
\syo{along} the coordinate directions, one can always increase or decrease each of the $f_i$'s. \moh{Using} a similar argument, this result also holds for $\lambda<0$. So, the 
optimum for each $f_i$ is \moh{attained} at some vertices of the interval 
$\mathbb{I}\mathcal{Z}$.  }
\end{proof} 

\subsection{\sy{Modeling Framework and} Problem Statement}
We consider constrained dynamical systems of the form:
 \begin{align} \label{eq:main_sys}
x_t^+=\syo{\tilde{f}(x_t,u_t,w_t)\triangleq f(z_t),} \quad 
\syo{\mu(x_t,u_t)\triangleq \nu(x_t) \in \mathcal{Y}_t,}  
\end{align}
where $x_t^+=\dot{x}_t$ if \eqref{eq:main_sys} is a continuous-time \syo{system} and $x_t^+=x_{t+1}$ if \eqref{eq:main_sys} is a discrete-time system. 
$x_t \in \mathcal{X}
\subset \mathbb{R}^{n_x}$ is the state, \sy{$u_t \in \mathcal{U} \subset \mathbb{R}^{n_x}$ is the known input} and \syo{$w_t \in \mathcal{W} \subseteq \mathbb{I}\mathcal{W} \triangleq  [\underline{w} , \overline{w}]  \in \mathbb{IR}^{n_w}$} 
is the augmentation of all 
exogenous inputs, e.g., 
bounded disturbance/noise and internal uncertainties such as uncertain parameters, with known bounds $\underline{w},\overline{w}$, while 
$f 
: \mathcal{X} \times \mathcal{U} \times \mathcal{W}  \to \mathbb{R}^{n_x}$ \sy{and $\mu 
:  
\mathcal{X} \times \mathcal{U} 
\to \mathbb{R}^{n_\syo{y}}$} \sy{are} the nonlinear vector field and the observation/constraint mapping, respectively, \sy{while $\mathcal{Y}_t$  is the known or measured time-varying constraint/observation set.}  
\sy{The mapping $\mu(\cdot)$ and the set $\mathcal{Y}_t \subseteq [\underline{y}_t,\overline{y}_t]$ describe system constraints that can represent prior or additional knowledge about the system states, e.g., sensor observations or measurements with bounded noise, known state constraints, manufactured constraints from modeling redundancy \cite{shen2017rapid} (cf. Section \ref{ex:redundancy} for an example), etc.} 
\syo{For ease of exposition, we further define $f : \mathcal{Z} \triangleq \mathcal{X} \times \mathcal{W} 
\subset \mathbb{R}^{n_z} \to \mathbb{R}^{n_x}$ and $\nu : \mathcal{X}
\subset \mathbb{R}^{n_x} \to \mathbb{R}^{n_y}$ as in \eqref{eq:main_sys} that is implicitly dependent on $u_t$ with the augmented state $z_t\triangleq [x_t^\top  \ w_t^\top]^\top \in \mathcal{Z}$.} 
Further, we assume the following. 
\begin{assumption}\label{assum:ELLC}
The mappings $f$ and  \syo{$\nu$} 
are either-sided locally Lipschitz (ELLC; cf. Definition \ref{defn:locally-Lip}).
\end{assumption}
\begin{assumption}\label{assum:jacobian-bounds}
\sy{For the mappings $f$ and \syo{$\nu$}, there exist known bounds on their Clarke Jacobian matrices, $\overline{J}^f_C,\underline{J}^f_C \in \{\mathbb{R} \cup \pm \infty \} ^{n_x \times n_z}$ and $\overline{J}^{\syo{\nu}}_C,\underline{J}^{\syo{\nu}}_C \in \{\mathbb{R} \cup \pm \infty \}^\syo{n_{y} \times n_{x}}$, that satisfy 
 \begin{align*}
\hspace{-0.1cm}\begin{array}{c}
((J^f_C(z))_{ij}\leq ({\overline{J}^f_C})_{ij} < \infty) \vee ((J^f_C(z))_{ij}\geq ({\underline{J}^f_C})_{ij}>-\infty), \\
 ((J^{\syo{\nu}}_C(\syo{x}))_{pq}\leq ({\overline{J}^{\syo{\nu}}_C})_{pq} < \infty) \vee ((J^{\syo{\nu}}_C(\syo{x}))_{pq}\geq ({\underline{J}^{\syo{\nu}}_C})_{pq}\hspace{-0.05cm}> \hspace{-0.075cm}-\infty),
 \end{array}
 \end{align*}  
 $\forall z \in \mathbb{I}\mathcal{Z}$, $\forall  \syo{x\in \mathbb{I}\mathcal{X}}$, $\forall \syo{i,q} \in \mathbb{N}_{n_x}$, $\forall j \in \mathbb{N}_{n_z}$, $\forall p \in \mathbb{N}_{n_\syo{y}}$.} 
\end{assumption}
\sy{Under the above modeling framework and assumptions, this paper seeks to find}
tight and tractable (i.e., computable) remainder-form upper and lower decomposition functions and their induced inclusion functions (cf. Definitions \ref{def:inc_func}, \ref{defn:remainder_dec}, \ref{def:one_side_dec} and 
\sy{Proposition} \ref{cor:dec_inc}) 
as well as 
\sy{to develop} set inversion algorithms \sy{based on (mixed-monotone) decomposition functions}.

\sy{ \begin{problem}[Decomposition Functions]  \label{prob:construction}
Suppose Assumptions \ref{assum:ELLC} and \ref{assum:jacobian-bounds} hold. Construct and \sy{quantify the tightness (via the metric \eqref{eq:metric})} 
of remainder-form decomposition functions, \sy{by solving the following subproblems:} 
\begin{enumerate}[\one.1]
\item Given an 
ELLC vector field $f:\mathcal{Z}\subset \mathcal{R}^{n_z} 
\to \mathbb{R}^{n_x}$, 
construct a tractable family of   mixed-monotone remainder-form (i.e., additive) decomposition functions for $f(\cdot)$. \label{item:construct}  
\item Derive lower and upper bounds for the \emph{tightness} 
(quantified via \eqref{eq:metric}) of 
the family of remainder-form decomposition functions obtained in  1.\ref{item:construct}.\label{item:bounds} 
\item Find the tightest decomposition function(s) among the family of remainder-form decomposition functions 
obtained in 1.\ref{item:construct} and compare them with the decomposition function 
in \cite{yang2019sufficient} (cf. Proposition \ref{prop:Liren_dec}), and natural, centered and mixed-centered natural inclusions (cf. Proposition \ref{prop:natural_inclusions}).\label{prob:T_MM}
\end{enumerate}
\end{problem}
\begin{problem}[Set Inversion Algorithm] \label{prob:set-inv}
Suppose Assumptions \ref{assum:ELLC} and \ref{assum:jacobian-bounds} hold. Given a prior/propagated interval $\mathbb{I}\mo{\mathcal{X}}^p_\mo{t} \in \mathbb{IR}^{n_\mo{x}}$, a constraint/observation function and set, $\mu(\mo{x}_t,u_t)$ and $\mathcal{Y}_t\subseteq [\underline{y}_t,\overline{y}_t]$ with  known $u_t$, 
develop an algorithm 
\sy{to find} 
an interval superset of all $\mo{x}_t$ that are 
compatible with $\mu(\cdot)$, $[\underline{y}_t,\overline{y}_t]$ and $\mathbb{I}\mo{\mathcal{X}}^p_\mo{t}$, 
i.e., to find the updated/refined interval $\mathbb{I}\mo{\mathcal{X}}^u_t$ such that
\begin{align}\label{eq:set_inv}
\{\mo{x} \in \mathbb{I}\mo{\mathcal{X}}^p_\syo{t} \mid  
\mu(\mo{x},u_t) \in  [\underline{y}_t,\overline{y}_t]\} \subseteq  \mathbb{I}\mo{\mathcal{X}}^u_t\subseteq \mathbb{I}\mo{\mathcal{X}}^p_\mo{t}.
\end{align} 
\end{problem}}
\sy{In the context of constrained reachability analysis and interval observer design (cf. Section \ref{sec:application}), the solution of the generalized embedding system (cf. Definition \ref{def:embedding}) based on} 
the decomposition functions obtained from solving Problem \ref{prob:construction} 
provides the unconstrained reachable set (or propagated set), $\mathbb{I}\mo{\mathcal{X}}_t^p$, while the set inversion algorithm in Problem \ref{prob:set-inv} finds the constrained reachable set (or updated set) $\mathbb{I}\mo{\mathcal{X}}_t^u$. 
\section{Main Results}\label{sec:main}
To address the aforementioned problems, we first describe \sy{our proposed construction approach to find a tractable family of mixed-monotone remainder-form decomposition functions. Then, by characterizing their tightness, we can determine the tightest decomposition function among the proposed family. Further, 
we present a novel set inversion algorithm that serves as an alternative and improves on existing approaches, e.g., SIVIA in \cite[Chapter 3]{jaulinapplied}) and $\mathcal{I}_G$ in \cite[Algorithm 1]{yang2020accurate}).} 
\subsection{Remainder-Form Decomposition Functions}
\sy{To solve Problem 
1.\ref{item:construct}, we provide a constructive procedure for computing the family of remainder-form decomposition functions in a tractable manner (i.e., in closed-form). Intuitively, our approach is based on the idea of decomposing each ELLC function $f$ into the remainder/additive form, i.e., $f=g+h$, such that $g$ is a CJSS function (cf. Definition \ref{defn:Jac_sign_stab}; so that Proposition \ref{prop:CJSS_optimality} applies) 
by ``shifting'' the Clarke directional gradients of $f$ and accounting for the ``error'' using $h$. 
Since there are several ``shift'' directions, we obtain a family of decomposition functions.
%
Note that the construction below is to be independently performed for each dimension of the ELLC function $f$; hence, without loss of generality, we only consider a scalar ELLC function $f_i$ throughout this subsection.}

\sy{
\begin{thm}[A Family of Remainder-Form Decomposition Functions] \label{thm:fdf}
 Consider an ELLC vector field $f_i 
 : \mathcal{Z} 
 \subset \mathbb{R}^{n_z} \to \mathbb{R}$ 
 and suppose that Assumptions \ref{assum:ELLC} and \ref{assum:jacobian-bounds} hold. Then, $f_i(\cdot)$ admits 
 \sy{a} family of  mixed-monotone remainder-form decomposition functions \sy{denoted as $\{f_{d,i}(z,\hat{z};\mathbf{m},h(\cdot))\}_{\mathbf{m} \in \mathbf{M}_i, h(\cdot)\in \mathcal{H}_{\mathbf{M}_i}}$ that is parameterized by a supporting vector $\mathbf{m} \in \mathbf{M}_i$ and an ELLC remainder function $h (\cdot)\in \mathcal{H}_{\mathbf{M}_i}$, where}
\begin{align}\label{eq:decomp1} 
\hspace{-.1cm}f_{d,i}(z,\hat{z};\mathbf{m},h(\cdot))\hspace{-.05cm}=\hspace{-.05cm} h(\zeta_{\mathbf{m}}(\hat{z},{z}))\hspace{-.1cm}+\hspace{-.1cm}f_i(\zeta_{\mathbf{m}}(z,\hat{z})\hspace{-.05cm}) \hspace{-.05cm}-\hspace{-.05cm}h(\zeta_{\mathbf{m}}(z,\hat{z})),\hspace{-.1cm}
\end{align}
with   $\zeta_{\mathbf{m}}(z,\hat{z})\hspace{-0.1cm}=\hspace{-0.1cm}[\zeta_{\mathbf{m},1}(z,\hat{z}),\dots,\zeta_{\mathbf{m},n_z}(z,\hat{z})]^\top$\hspace{-0.05cm}, where $\forall j \hspace{-0.1cm}\in\hspace{-0.1cm} \mathbb{N}_{n_z}$,
\begin{align}\label{eq:corner} 
 \begin{array}{rl}
 &\zeta_{\mathbf{m},j}(z,\hat{z})=\begin{cases} \hat{z}_j, \ \text{if} \ \mathbf{m}_j \geq \max(({\overline{J}^f_C})_{ij},{0}), \\
z_j, \ \text{if} \ \mathbf{m}_j \leq \min(({\underline{J}^f_C})_{ij},{0}), 
 \end{cases}
 \end{array}
 \end{align} 
and the set of supporting vectors $\mathbf{M}_i$ is defined as:
 \begin{align}
    \begin{array}{rl}
 \mathbf{M}_i \triangleq & \hspace{-0.2cm}\{\mathbf{m} \in \mathbb{R}^ {n_z} \ | \  \mathbf{m}_j \geq \max(({\overline{J}^f_C})_{ij},{0})  \ \lor \\ 
  & \ \ \mathbf{m}_j \leq \min(({\underline{J}^f_C})_{ij},{0}),
 \forall j \in \mathbb{N}_{n_z} \}
    \end{array}\label{eq:Mj}
  \end{align}
  if \eqref{eq:main_sys} is a discrete-time system and with 
  \begin{align}
   \begin{array}{rl}
 &\zeta_{\mathbf{m},j}(z,\hat{z})=\begin{cases} \hat{z}_j, \ \text{if} \ \mathbf{m}_j \geq \max(({\overline{J}^f_C})_{ij},{0}), \\
z_j, \ \text{if} \ \mathbf{m}_j \leq \min(({\underline{J}^f_C})_{ij},{0}),\\
x_i, \ \text{if} \ j=i, 
 \end{cases}
 \end{array}\\
    \begin{array}{l}
  \mathbf{M}_i \triangleq \{\mathbf{m} \in \mathbb{R}^ {n_z} \ | \  \mathbf{m}_j \geq \max(({\overline{J}^f_C})_{ij},{0}) \ \lor \\ \quad \
  \mathbf{m}_j \leq \min(({\underline{J}^f_C})_{ij},{0}) ,
 \forall j \in \mathbb{N}_{n_z}, j \ne i, \mathbf{m}_i=0 \} 
    \end{array} \label{eq:Mj_CT}
  \end{align}
   if \eqref{eq:main_sys} is a continuous-time system, while the set of remainder functions $\mathcal{H}_{\mathbf{M}_i}$ is the family of all ELLC remainder functions whose Clarke sub-differential set over $\mathcal{Z}$ (cf. Definition \ref{defn:gen_grad}) is a subset of ${\mathbf{M}_i}$ given by:
    \begin{align} \label{eq:HM}
 \mathcal{H}_{\mathbf{M}_i} \hspace{-0.05cm}\triangleq\hspace{-0.05cm} \{{h}:\mathcal{Z} \to \mathbb{R} \ | \ [\underline{J}^{{h}}_C(z),\overline{J}^{{h}}_C(z)] \subseteq \mathbf{M}_i , \forall z \in \mathcal{Z} \}.
 \end{align}
 The resulting family of decomposition-based inclusion functions is given by:
 \begin{align} \label{eq:TR}
 T^{f_{d,i}}_{\underline{\mathbf{m}},\overline{\mathbf{m}}}(\mathbb{I}\mathcal{Z}) \triangleq [{f}_{d,i}(\underline{z},\overline{z};\underline{\mathbf{m}},\underline{h}(\cdot)),{f}_{d,i}(\overline{z},\underline{z};\overline{\mathbf{m}},\overline{h}(\cdot))] , 
\end{align} 
for all $\underline{\mathbf{m}},\overline{\mathbf{m}}\in \mathbf{M}_i$, 
and the corresponding $\underline{h}(\cdot),\overline{h}(\cdot)\in \mathcal{H}_{\mathbf{M}_i}$.  
 \end{thm}}
 
 It may be worth noting that the small difference in the definitions of 
$\mathbf{M}$ for discrete-time versus continuous-time systems in Theorem \ref{thm:fdf} 
originates from the subtle difference between the definitions of decomposition functions for discrete-time and continuous-time systems (cf. Definition \ref{defn:dec_func}). 
 \sy{To prove the above theorem, we first prove the following lemma.
 
\begin{lem} \label{lem:gh_aligned}
 Any remainder function $h(\cdot) \in \mathcal{H}_{\mathbf{M}_i}$ (cf. \eqref{eq:HM}) is CJSS and the function $g_i(\cdot) \triangleq f_i(\cdot)-h(\cdot)$ is also CJSS. Moreover, the pair $(g_i(\cdot),-h(\cdot))$ is aligned, 
 i.e., $\forall j \in \mathbb{N}_{n_z}$, 
 \begin{enumerate}[(i)]
 \item $(\overline{J}^{g_i}_C(z))_{j}\leq 0$ \syo{if and only if} 
 $(\underline{J}^h_C(z))_{j}\geq 0$, or
 \item $(\underline{J}^{g_i}_C(z))_{j}\geq 0$ \syo{if and only if}  
 $(\overline{J}^h_C(z))_{j}\leq 0$.
 \end{enumerate}
 \end{lem}

\begin{proof} Since $\mathbf{m} \in \mathbf{M}_i$, by construction of $\mathbf{M}_i$ (cf. \syo{\eqref{eq:Mj}, \eqref{eq:Mj_CT}}), $h(\cdot) \in \mathcal{H}_{\mathbf{M}_i}$ is a CJSS function. Next, by applying Proposition \ref{prop:clarke-sum} to $g_i(\cdot) \triangleq f_i(\cdot)-h(\cdot)$, we know that 
$\forall z \in \mathcal{Z},
 \forall j \in \mathbb{N}_{n_z}$: 
\begin{align}
\label{eq:CJSS-} 
(\underline{J}^{g_i}_C(z))_{j} &\leq (\underline{J}^{f}_C(z))_{ij}-(\underline{J}^{h}_C(z))_{j}, \\ \relax
\label{eq:CJSS+}
(\overline{J}^{g_i}_C(z))_{j} &\geq (\overline{J}^{f}_C(z))_{ij}-(\overline{J}^{h}_C(z))_{j}.
\end{align} 
Then, since $h(\cdot) \in \mathcal{H}_{\mathbf{M}_i}$, according to \syo{\eqref{eq:corner}}
--\eqref{eq:HM}, 
we consider the following two cases (the case with $\mathbf{m}_i=0$ for continuous-time systems is trivial):
\begin{enumerate}[(i)]
\item $\forall z \in \mathcal{Z},(\underline{J}^{h}_C(z))_{j} \geq  \max(({\overline{J}^f_C})_{ij},{0}) \geq 0$. 
Then, from \eqref{eq:CJSS-}, we obtain 
$\forall z \in \mathcal{Z},(\underline{J}^{g_i}_C(z))_{j} \leq 0 \Rightarrow {(\underline{J}^{g_i}_C})_{j} \leq 0$.  
\item $\forall z \in \mathcal{Z},(\overline{J}^{h}_C(z))_{j} \leq  \min(({\underline{J}^f_C})_{ij},{0}) \leq 0$. Then, from 
\eqref{eq:CJSS+}, we have 
$\forall z \in \mathcal{Z},(\overline{J}^{g_i}_C(z))_{j} \geq 0 \Rightarrow ({\overline{J}^{g_i}_C})_{j} \geq 0$.
\end{enumerate}
The reverse can be similarly deduced. Finally, since $({\overline{J}^{g_i}_C})_{j} \geq 0$ or ${(\underline{J}^{g_i}_C})_{j} \leq 0$ holds, $g_i(\cdot)$ is CJSS by Definition \ref{defn:Jac_sign_stab}.
\end{proof}

\begin{rem}
Since the pair $(g_i,-h)$ is aligned 
and $f_i=g_i+h$, the proposed remainder-form decomposition function can also be viewed as the decomposition of $f_i$ into a \emph{difference of monotone functions}, which is similar in spirit to difference of convex functions in DC programming, e.g., \cite{alamo2008set,horst1999dc}.
\end{rem}}


 \begin{proof}[Proof of Theorem \ref{thm:fdf}]
\sy{ Armed by Lemma \ref{lem:gh_aligned}, we now prove that \eqref{eq:decomp1} is mixed-monotone (cf.  Definition \ref{defn:dec_func}). Having defined $g_i(\cdot) \triangleq f_i(\cdot)-h(\cdot)$, \eqref{eq:decomp1} can be rewritten as
  \begin{align}\label{eq:decomposition2}
f_{d,i}(z,\hat{z};\mathbf{m},h)= h(\zeta_{\mathbf{m}}(\hat{z},{z}))+g_i(\zeta_{\mathbf{m}}({z},\hat{z})). 
 \end{align}
 First, it directly follows from \eqref{eq:corner} that $\zeta_{\mathbf{m}}(z,{z})=z$ and $f_{d,i}(z,{z};\mathbf{m},h(\cdot))=f_i(z)$. Hence, it remains to show that  
$f_{d,i}(z,\hat{z};\mathbf{m},h(\cdot))$ 
 is non-decreasing in $z$ and non-increasing in $\hat{z}$. To do this, consider
$z,\tilde{z},\hat{z} \in \mathcal{Z}$, where $\tilde{z} \geq z$. 
Let $j_0 \in \mathbb{N}_{n_z}$, and suppose case (i) in the proof of Lemma \ref{lem:gh_aligned} holds for dimension $j_0$. Then, by the  first case in  \eqref{eq:corner},
\begin{align}\label{eq:h_corner}
\begin{array}{c}
\zeta_{\mathbf{m},j_0}(\hat{z},\tilde{z})=\tilde{z}_{j_0} \geq {z}_{j_0}=\zeta_{\mathbf{m},j_0}(\hat{z},z), \\
\zeta_{\mathbf{m},j_0}(\tilde{z},\hat{z})=\hat{z}_{j_0}=\zeta_{\mathbf{m},j_0}(z,\hat{z}).
 \end{array}
  \end{align}
   Next, we define ${z}^1 \in \mathcal{Z}$ as follows: ${z}^1_{j_0}=\tilde{z}_{j_0}$ and ${z}^1_{j}=z_j, \forall j \ne j_0$. Thus, $z^1\ge z$, and by \eqref{eq:h_corner}, $\zeta_{\mathbf{m}}(z^1,\hat{z})=\zeta_{\mathbf{m}}(z,\hat{z})$,  $\zeta_{\mathbf{m},j}(\hat{z},z^1) = \zeta_{\mathbf{m},j}(\hat{z},z), \forall j\neq j_0$ and $\zeta_{\mathbf{m},j_0}(\hat{z},z^1) \ge \zeta_{\mathbf{m},j_0}(\hat{z},z)$.
 Moreover, by case (i) in the proof of Lemma \ref{lem:gh_aligned} and Proposition \ref{prop:CJSS_optimality}, $h(\cdot)$ is non-decreasing in the dimension $j_0$ and thus, $h(\zeta_{\mathbf{m}}(\hat{z},{z}^1))\ge h(\zeta_{\mathbf{m}}(\hat{z},{z}))$ and $g_i(\zeta_{\mathbf{m}}({z}^1,\hat{z}))=g_i(\zeta_{\mathbf{m}}({z},\hat{z}))$. Then, it follows from \eqref{eq:decomposition2} that
$f_{d,i}(z^1,\hat{z};\mathbf{m},h(\cdot))
\geq 
f_{d,i}(z,\hat{z};\mathbf{m},h(\cdot))$.
  Repeating this procedure sequentially for all dimensions $j$ for which case (i) in Lemma \ref{lem:gh_aligned} holds (where $\tau$ is the size of this set), we obtain: 
  \begin{gather}\label{eq:ineq1}
  \begin{array}{c}
  f_{d,i}(z^\tau \hspace{-.cm}, \hspace{-.0cm}\hat{z};\mathbf{m},h(\cdot)) \hspace{-.cm}\geq  \hspace{-.cm} f_{d,i}(z^{\tau-1} \hspace{-.cm},\hspace{-.0cm}\hat{z};\mathbf{m},h(\cdot))  
  \hspace{-.cm}\geq \hspace{-.cm} \dots  \hspace{-.cm}\\
\hspace{1.25cm}  \geq \hspace{-.cm} f_{d,i}(z^1 \hspace{-.cm},\hspace{-.05cm}\hat{z};\mathbf{m},h(\cdot))  \hspace{-.cm}\geq  \hspace{-.cm}f_{d,i}(z,\hat{z};\mathbf{m},h(\cdot)).
\end{array}
  \end{gather}

Next, we consider the rest of the dimensions $j'$ that satisfy case (ii) in Lemma \ref{lem:gh_aligned}. It follows from the second case in \eqref{eq:corner}  that for such a dimension $j_0'\in\mathbb{N}_{n_z}$, 
%
  \begin{align}\label{eq:h_corner2}
\begin{array}{c}
 \zeta_{\mathbf{m},{j'_0}}(\hat{z},\tilde{z})=\hat{z}_{j'_0}=\zeta_{\mathbf{m},{j'_0}}(\hat{z},z^\tau), \\
\zeta_{\mathbf{m},{j'_0}}(\tilde{z},\hat{z})=\tilde{z}_{j'_0} \geq {z}^\tau_{j'_0}=\zeta_{\mathbf{m},{j'_0}}(z^\tau,\hat{z}).
 \end{array}
  \end{align}
  Repeating a similar procedure as for case (i), we define ${z}^{\tau+1} \in \mathcal{Z}$ as follows: ${z}^{\tau+1}_{j'_0}=\tilde{z}_{j'_0}$ and ${z}^{\tau+1}_{j}=z^\tau_j, \forall j \ne j'_0$. Thus, $z^{\tau+1}\ge z^\tau$, and by \eqref{eq:h_corner}, $\zeta_{\mathbf{m}}(\hat{z},z^{\tau+1})=\zeta_{\mathbf{m}}(\hat{z},z^\tau)$,  $\zeta_{\mathbf{m},j}(z^{\tau+1},\hat{z}) = \zeta_{\mathbf{m},j}(z^\tau,\hat{z}), \forall j\neq j_0'$ and $\zeta_{\mathbf{m},j_0'}(z^{\tau+1},\hat{z}) \ge \zeta_{\mathbf{m},j_0'}(z^\tau,\hat{z})$.
 Moreover, by case (ii) in the proof of Lemma \ref{lem:gh_aligned} and Proposition \ref{prop:CJSS_optimality}, $g_i(\cdot)$ is non-decreasing in the dimension $j_0'$ and thus, $g_i(\zeta_{\mathbf{m}}({z}^{\tau+1},\hat{z}))\ge g_i(\zeta_{\mathbf{m}}({z}^\tau,\hat{z}))$ and $h(\zeta_{\mathbf{m}}(\hat{z},{z}^{\tau+1}))=h(\zeta_{\mathbf{m}}(\hat{z},{z}^\tau))$. Then, it follows from \eqref{eq:decomposition2} that
$f_{d,i}(z^{\tau+1},\hat{z};\mathbf{m},h(\cdot))
\geq 
f_{d,i}(z^\tau,\hat{z};\mathbf{m},h(\cdot))$.

Repeating this procedure sequentially for all dimensions $j$ for which case (ii) in Lemma \ref{lem:gh_aligned} holds, we obtain: 
  \begin{align}\label{eq:ineq2}
  \begin{array}{c}
  f_{d,i}(z^\tau,\hat{z};\mathbf{m},h(\cdot)) \leq f_{d,i}(z^{\tau+1},\hat{z};\mathbf{m},h(\cdot)) \leq \dots \\
  \hspace{0.75cm} \leq f_{d,i}(\tilde{z},\hat{z};\mathbf{m},h(\cdot)), 
  \end{array}
  \end{align}
  where the last term is 
  $ f_{d,i}(\tilde{z},\hat{z};\mathbf{m},h(\cdot))$ since there exist only two possible cases (i) or (ii) for each dimension. Combining \eqref{eq:ineq1} and \eqref{eq:ineq2} yields $ f_{d,i}(\tilde{z},\hat{z};\mathbf{m},h(\cdot)) \geq  f_{d,i}({z},\hat{z};\mathbf{m},h(\cdot))$, which means that $f_{d,i}$ is non-decreasing in its first argument. A\syo{n almost identical} 
  argument shows that $f_{d,i}$ is non-increasing in its second argument. 
Thus, $f_{d,i}$ is mixed-monotone.}
 \end{proof} 

Theorem \ref{thm:fdf} mathematically introduces a family of decomposition functions (cf. Definition \ref{defn:dec_func}), but the results are not yet tractable (cf. Definition \ref{defn:tractability}), since to build such a family, we have to search over $\mathbf{M}_i$ (cf. \eqref{eq:Mj} and \eqref{eq:Mj_CT}), which is an unbounded and infinite set, as well as over $\mathcal{H}_{\mathbf{M}_i}$ (cf. \eqref{eq:HM}), which is an infinite-dimensional space of functions. 
To overcome this problem, we propose tractable upper and lower decomposition functions that only need to search over a \emph{finite} set of supporting vector $\mathbf{M}_i^c \subset \mathbf{M}_i$ with the choice of linear remainder functions $h(\zeta)=\langle \mathbf{m},\zeta\rangle = \mathbf{m}^\top \zeta$, and prove that these tractable decomposition functions are the tightest among the family of decomposition functions in Theorem \ref{thm:fdf}.

\begin{thm}[Tight and Tractable Remainder-Form Upper and Lower Decomposition Functions] \label{thm:tractableDF}
 Consider an ELLC vector field $f_i 
 : \mathcal{Z} 
 \subset \mathbb{R}^{n_z} \to \mathbb{R}$ 
 and let Assumptions \ref{assum:ELLC} and \ref{assum:jacobian-bounds} hold. Then, the tightest tractable (mixed-monotone) remainder-form upper and lower decomposition functions with $z \ge \hat{z}$ are 
 \begin{align}\label{eq:UL_DF}
 \begin{array}{l}
\overline{f}_{d,i}(z,\hat{z})= 
\min \limits_{\mathbf{m} \in \mathbf{M}^c} f_i(\zeta_{\mathbf{m}}^+) +\mathbf{m}^\top(\zeta_{\mathbf{m}}^- -\zeta_{\mathbf{m}}^+),\\
\underline{f}_{d,i}(\hat{z},z)= 
\max \limits_{\mathbf{m} \in \mathbf{M}^c} f_i(\zeta_{\mathbf{m}}^-)+\mathbf{m}^\top(\zeta_{\mathbf{m}}^+ -\zeta_{\mathbf{m}}^-),
\end{array}
\end{align}
where $\zeta_{\mathbf{m}}^+\triangleq \zeta_{\mathbf{m}}(z,\hat{z})$ and $\zeta_{\mathbf{m}}^-\triangleq \zeta_{\mathbf{m}}(\hat{z},z)$  with   $\zeta_{\mathbf{m}}(\cdot,\cdot)$ defined in \eqref{eq:corner} 
and the \emph{finite} set of supporting vectors $\mathbf{M}_i^c$ defined as:
 \begin{align}
    \begin{array}{rl}
 \mathbf{M}_i^c \triangleq & \hspace{-0.2cm}\{\mathbf{m} \in \mathbb{R}^ {n_z} \ | \ \mathbf{m}_j = \max(({\overline{J}^f_C})_{ij},{0})  \ \lor \\ 
  & \ \mathbf{m}_j = \min(({\underline{J}^f_C})_{ij},{0}),
 \forall j \in \mathbb{N}_{n_z} \}
    \end{array}\label{eq:Mc}
  \end{align}
  if \eqref{eq:main_sys} is a discrete-time system and
  \begin{align}\label{eq:Mc_CT}
    \begin{array}{l}
  \mathbf{M}_i^c \triangleq \{\mathbf{m} \in \mathbb{R}^ {n_z} \ | \ \mathbf{m}_j = \max(({\overline{J}^f_C})_{ij},{0})  \ \lor \\ \quad \
 \mathbf{m}_j =\min(({\underline{J}^f_C})_{ij},{0}),
 \forall j \in \mathbb{N}_{n_z}, j \ne i, \mathbf{m}_i=0 \} 
    \end{array}
  \end{align}   if \eqref{eq:main_sys} is a continuous-time system.
  
  Moreover, we call the resulting inclusion function for an interval domain $\mathbb{I}\mathcal{Z}=[\underline{z},\overline{z}]$ the remainder-form inclusion function $T^{f_{d,i}}_R\triangleq [\underline{f}_{d,i}(\underline{z},\overline{z}),\overline{f}_{d,i}(\overline{z},\underline{z})]$ (cf. Definition \ref{defn:remainder_dec}).
 \end{thm}
 
 We will prove the above theorem in two steps with the help of the following lemmas, where the two steps show that we can restrict our search for the tightest upper and lower decomposition functions to a \emph{finite} set of supporting vectors and the set of \emph{linear} remainder functions, respectively, without introducing any conservatism.

\begin{lem}[Finite Set of Supporting Vectors] \label{lem:upper_lowwer_df}
Suppose the assumptions in Theorem \ref{thm:tractableDF} hold. Then, $\forall z,\hat{z} \in \mathcal{Z}$, $\forall h(\cdot) \in \mathcal{H}_{\mathbf{M}_i}$ and for both $\opt \in \{\min,\max\}$, 
\begin{align*}
\opt\limits_{\substack{\mathbf{m} \in \mathbf{M}_i,\\ h(\cdot) \in \mathcal{H}_{\mathbf{M}_i}}}  f_{d,i}(z,\hat{z};\mathbf{m},h(\cdot))=  \opt \limits_{\substack{\mathbf{m} \in \mathbf{M}_i^c,\\ h(\cdot) \in \mathcal{H}_{\mathbf{M}_i}} } f_{d,i}(z,\hat{z};\mathbf{m},h(\cdot)), 
\end{align*}
where $f_{d,i}(z,\hat{z};\mathbf{m},h(\cdot))$ is defined in \eqref{eq:decomp1}, $\mathbf{M}_i$ in 
\eqref{eq:Mj} or \eqref{eq:Mj_CT},  $\mathbf{M}_i^c$ in 
\eqref{eq:Mc} or \eqref{eq:Mc_CT}, and $\mathcal{H}_{\mathbf{M}_i}$ in \eqref{eq:HM}. 
 \end{lem}
 \begin{proof}
 We consider $\mathbf{m} \in \mathbf{M}_i$ and construct $\tilde{\mathbf{m}} \in \mathbf{M}_i^c$ as follows for all $j \in \mathbb{N}_{n_z}$: $\tilde{\mathbf{m}}_j =\max(({\overline{J}^{f}_C})_{ij},0)$ if $\mathbf{m}_j  \geq \max(({\overline{J}^{f}_C})_{ij},0)$ and $\tilde{\mathbf{m}}_j =\min(({\underline{J}^{f}_C})_{ij},0)$ if $\mathbf{m}_j  \leq \min(({\underline{J}^{f}_C})_{ij},0)$. Then, it can be easily verified from \eqref{eq:corner} that $\forall z^1, z^2 \in \mathcal{Z}, {\zeta}_{\mathbf{m}}(z^1,z^2)={\zeta}_{\tilde{\mathbf{m}}}(z^1,z^2)$, and hence by \eqref{eq:decomp1}, $f_{d,i}(z^1,z^2;\mathbf{m},h(\cdot))=f_{d,i}(z^1,z^2;\tilde{\mathbf{m}},h(\cdot))\mo{,\forall h(\cdot) \in \mathcal{H}_{\mathbf{M}_i}}$. Hence, for any $\mathbf{m} \in \mathbf{M}_i$, there exists $\tilde{\mathbf{m}} \in \mathbf{M}_i^c$ that admits an equivalent decomposition function and correspondingly, the optimization over $\mathbf{M}_i$ and $\mathbf{M}_i^c$ are equivalent.  
%
 \end{proof}

\begin{lem}[Linear Remainder Functions] \label{lem:mm_dec} 
Suppose the assumptions in Theorem \ref{thm:tractableDF} hold. Then, $\forall z,\hat{z} \in \mathcal{Z}, z\ge \hat{z}$, 
\begin{align}
 \hspace{-0.15cm}\begin{array}{l}
 \min \limits_{\substack{\mathbf{m} \in \mathbf{M}_i^c, \\ h(\cdot) \in \mathcal{H}_{\mathbf{M}_i}}}\hspace{-0.1cm} f_{d,i}(z,\hat{z};\mathbf{m},h(\cdot))= 
\min \limits_{\mathbf{m} \in \mathbf{M}^c} f_i(\zeta_{\mathbf{m}}^+)+\mathbf{m}^\top(\zeta_{\mathbf{m}}^- -\zeta_{\mathbf{m}}^+), 
 \\
 \max \limits_{\substack{\mathbf{m} \in \mathbf{M}_i^c, \\ h(\cdot) \in \mathcal{H}_{\mathbf{M}_i}}} \hspace{-0.1cm} f_{d,i}(\hat{z},z;\mathbf{m},h(\cdot)) = 
\max \limits_{\mathbf{m} \in \mathbf{M}^c} f_i(\zeta_{\mathbf{m}}^-)+\mathbf{m}^\top(\zeta_{\mathbf{m}}^+-\zeta_{\mathbf{m}}^-),\vspace{-0.2cm}
\end{array}\label{eq:linearopt}
\end{align}
where $\zeta_{\mathbf{m}}^+\triangleq \zeta_{\mathbf{m}}(z,\hat{z})$ and $\zeta_{\mathbf{m}}^-\triangleq \zeta_{\mathbf{m}}(\hat{z},z)$ with $f_{d,i}(z,\hat{z};\mathbf{m},h(\cdot))$ is defined in \eqref{eq:decomp1}, $\zeta_{\mathbf{m}}(\cdot,\cdot)$ defined in \eqref{eq:corner}, $\mathbf{M}_i^c$ in 
\eqref{eq:Mc} or \eqref{eq:Mc_CT}, and $\mathcal{H}_{\mathbf{M}_i}$ in \eqref{eq:HM}. 
%
\end{lem}
\begin{proof}
Consider any $h(\cdot)\in \mathcal{H}_{\mathbf{M}_i}$, $\mathbf{m}\in \mathbf{M}_i$ and from \eqref{eq:decomp1},
\begin{align}
f^+_{d,i}&\triangleq f_{d,i}(z,\hat{z};\mathbf{m},h(\cdot))=f_i(\zeta_\mathbf{m}^+)+\Delta h_\mathbf{m}, \label{eq:fd+}\\
 f^-_{d,i}&\triangleq f_{d,i}(\hat{z},z;\mathbf{m},h(\cdot))=f_i(\zeta_\mathbf{m}^-)-\Delta h_\mathbf{m}, \label{eq:fd-}
\end{align} 
where $\Delta h_\mathbf{m} \triangleq h(\zeta_{\mathbf{m}}^-)-h(\zeta_{\mathbf{m}}^+)$.  Then, applying the Clarke Mean Value Theorem \cite[Chapter II, Theorem 1.3]{demianov1995constructive} to $\Delta h_\mathbf{m}$, 
there exists $\xi \in [{\underline{J}^{h}_C},{\overline{J}^{h}_C}] \subset \mathbf{M}_i$ such that
 \begin{align}\label{eq:mvt}
 \Delta h_{\mathbf{m}}= \langle \xi,(\zeta_{\mathbf{m}}^+ - \zeta_{\mathbf{m}}^-)\rangle.
 \end{align}
Since $\xi \in \mathbf{M}_i$, by \eqref{eq:Mj} and \eqref{eq:Mj_CT}, we know that $ \xi_j \leq \min(({\underline{J}^f_C})_{ij},{0})$ or $\xi_j \geq \max(({\overline{J}^f_C})_{ij},{0})$, $\forall j \in \mathbb{N}_{n_z}$ (excluding $j=i$ for continuous-time systems where $\xi_j=0$).  

 Then, similar to the proof of Lemma \ref{lem:gh_aligned}, $\forall j \in \mathbb{N}_{n_x}$, we can consider two cases corresponding to the two cases in \eqref{eq:corner}: 
 \begin{enumerate}[(i)]
\item $\xi_{j} \geq  \max(({\overline{J}^f_C})_{ij},{0}) \geq 0$: 
From \eqref{eq:corner}, $\zeta_{\mathbf{m},j}^-=\zeta_{\mathbf{m},j}(\hat{z},{z})={z}_j$, $\zeta_{\mathbf{m},j}^+=\zeta_{\mathbf{m},j}(z,\hat{z})=\hat{z}_j$ and ${z}_j\ge \hat{z}_j$, thus, we have $\zeta_{\mathbf{m},j}^- - \zeta_{\mathbf{m},j}^+ \ge 0$ and   $\xi_{j}(\zeta_{\mathbf{m},j}^- - \zeta_{\mathbf{m},j}^+) \geq  \max(({\overline{J}^f_C})_{ij},{0})(\zeta_{\mathbf{m},j}^- - \zeta_{\mathbf{m},j}^+) $. Then, the minimum of $f^+_{d,i}$ in \eqref{eq:fd+} and the maximum of $f^-_{d,i}$ in \eqref{eq:fd-} are attained in \eqref{eq:linearopt} when $\xi_j=\max(({\overline{J}^f_C})_{ij},{0}) \in (\mathbf{M}^c_i)_j$.

\item $\xi_j \leq  \min(({\underline{J}^f_C})_{ij},{0}) \leq 0$: From \eqref{eq:corner}, $\zeta_{\mathbf{m},j}^-=\zeta_{\mathbf{m},j}(\hat{z},{z})=\hat{z}_j$, $\zeta_{\mathbf{m},j}^+=\zeta_{\mathbf{m},j}(z,\hat{z})={z}_j$ and ${z}_j\ge \hat{z}_j$, thus, we have $\zeta_{\mathbf{m},j}^- - \zeta_{\mathbf{m},j}^+ \le 0$ and   $\xi_{j}(\zeta_{\mathbf{m},j}^- - \zeta_{\mathbf{m},j}^+) \geq  \min(({\underline{J}^f_C})_{ij},{0})(\zeta_{\mathbf{m},j}^- - \zeta_{\mathbf{m},j}^+) $. Then, the minimum of $f^+_{d,i}$ in \eqref{eq:fd+} and the maximum of $f^-_{d,i}$ in \eqref{eq:fd-} are attained in \eqref{eq:linearopt} when $\xi_j=\min(({\underline{J}^f_C})_{ij},{0}) \in (\mathbf{M}^c_i)_j$.
\end{enumerate}
 
Finally, we can restrict our search to the class of linear remainder functions $h(\zeta)=\langle \mathbf{m}, \zeta \rangle = \mathbf{m}^\top\zeta$ with $\mathbf{m}\in \mathbf{M}^c_i$, since it can achieve the optima in \eqref{eq:linearopt}. 
\end{proof}

\begin{proof}[Proof of Theorem \ref{thm:tractableDF}] First, by repeatedly applying Corollary \ref{cor:min_max_dec} on all the decomposition functions in the family \eqref{eq:decomp1} and the fact that the upper and lower decomposition functions can be optimized independently, it can be seen that the tightest upper and lower decomposition functions with $z,\hat{z}\in \mathcal{Z}, z \ge \hat{z}$ are 
 \begin{align*}
 \begin{array}{l}
\overline{f}_{d,i}(z,\hat{z})= \min\limits_{\substack{\mathbf{m} \in \mathbf{M}_i, h(\cdot) \in \mathcal{H}_{\mathbf{M}_i}}}  f_{d,i}(z,\hat{z};\mathbf{m},h(\cdot)),\\
\underline{f}_{d,i}(\hat{z},z)=\max\limits_{\substack{\mathbf{m} \in \mathbf{M}_i, h(\cdot) \in \mathcal{H}_{\mathbf{M}_i}}}  f_{d,i}(\hat{z},z;\mathbf{m},h(\cdot)),
\end{array}
\end{align*}
Then, by Lemmas \ref{lem:upper_lowwer_df} and \ref{lem:mm_dec}, we obtain the tractable and tight upper and lower decomposition functions in \eqref{eq:UL_DF}.
\end{proof}

Theorem \ref{thm:tractableDF} guarantees that in order to obtain the tightest possible decomposition function in the form of \eqref{eq:decomp1}, it is sufficient to only search over a \emph{finite} set of supporting vectors $\mathbf{M}_i^\mo{c}$ and the class of \emph{linear} remainder functions with gradients from this same set $\mathbf{M}_i^\mo{c}$, i.e., $h(\zeta)=\moh{\langle}\mathbf{m},\zeta\moh{\rangle}=\mathbf{m}^\top \zeta, \forall\mathbf{m} \in \mathbf{M}_i^c$, where the search space is the finite and countable set $\mathbf{M}^c_i$. 
Hence, the optimal search for the tightest decomposition functions is computable/tractable according to Definition \ref{defn:tractability}. 

Further, the result in Theorem \ref{thm:tractableDF} can be applied to each $f_i$, $i \in \mathbb{N}_{n_x}$ of the function $f$ to obtain the tightest remainder-form decomposition functions from the family of \emph{remainder-form CJSS decomposition functions} \syo{in} 
\eqref{eq:decomp1}. This is summarized in Algorithm \ref{algorithm1}, which takes an interval domain $\mathbb{I}\mathcal{Z}=[\underline{z},\overline{z}]$, the function $f$ and its Clarke \syo{Jacobians,} 
$\overline{J}_C^f$, $\underline{J}_C^f$, as inputs, and \syo{outputs} 
the remainder-form inclusion function $T^{f_{d}}_R\triangleq [\underline{f}_{d}(\underline{z},\overline{z}),\overline{f}_{d}(\overline{z},\underline{z})]$ (cf. Definition \ref{defn:remainder_dec}).

\begin{algorithm}[t] \small
\caption{Remainder-Form Decomposition Functions}\label{algorithm1}
\begin{algorithmic}[1]
\Function{$T^{f_d}_{R}$}{$f(\cdot),\overline{J}_C^f,\underline{J}_C^f,\overline{z},\underline{z}$}
\State Initialize: $\overline{f}_{d} \gets\infty,\underline{f}_{d} \gets-\infty$; 
\For{$i=1$ to $n_x$}
		\If{\eqref{eq:main_sys} is a discrete-time system}
		 \State \hspace{0cm} $\mathbf{M}_i^c \hspace{-.1cm}  \triangleq  \hspace{-.1cm}\{\mathbf{m} \in \mathbf{M}_i | \ \mathbf{m}_j = \min(({\underline{J}_C^f})_{ij},0) \ \vee$              
		\Statex \hspace{2.5cm} $\mathbf{m}_j= \max(({\overline{J}_C^f})_{ij},{0}), \forall j \in \mathbb{N}_{n_z} \}$.

		\EndIf
		\If{\eqref{eq:main_sys} is a continuous-time system}
		 \State \hspace{0cm} $\mathbf{M}_i^c \hspace{-.1cm}  \triangleq  \hspace{-.1cm}\{\mathbf{m} \in \mathbf{M}_i | \ \mathbf{m}_j = \min(({\underline{J}_C^f})_{ij},0) \ \vee \mathbf{m}_j=$              
		\Statex \hspace{1.8cm} $ \max(({\overline{J}_C^f})_{ij},{0}), \forall j \in \mathbb{N}_{n_z}$, $j \ne i$ and $\mathbf{m}_i=0\}$.

		\EndIf
				\For {$\mathbf{m} \in \mathbf{M}^c_i$} 
		\For{$j=1$ to $n_z$}
		\If{\eqref{eq:main_sys} is a continuous-time system $\land (j=i)$}
		\State \hspace{0cm} $\zeta^+_{\mathbf{m},j} \gets x_i$; \ $\zeta^-_{\mathbf{m},j} \gets x_i$;
		\ElsIf{$\mathbf{m}_j = \min(({\underline{J}_C^f})_{ij},{0})$}
		\State \hspace{0cm} $\zeta^+_{\mathbf{m},j} \gets \overline{z}_{j}$; \ $\zeta^-_{\mathbf{m},j} \gets \underline{z}_{j}$;
		\Else 
		\State \hspace{0cm} $\zeta^+_{\mathbf{m},j} \gets \underline{z}_{j}$; \ $\zeta^-_{\mathbf{m},j} \gets \overline{z}_{j}$;
		\EndIf
		\State \hspace{0cm} $\overline{f}_{d,i} \gets\min(\overline{f}_{d,i},f_i(\zeta^+_\mathbf{m})+\mathbf{m}^\top(\zeta^-_\mathbf{m}-\zeta^+_\mathbf{m}))$;
		\State \hspace{0cm} $\underline{f}_{d,i}\gets\min(\underline{f}_{d,i},f_i(\zeta^-_\mathbf{m})+\mathbf{m}^\top(\zeta^+_\mathbf{m}-\zeta^-_\mathbf{m}))$;
		\EndFor
		\EndFor
		
\EndFor
 \State \Return $\overline{f}_{d},\underline{f}_{d}$;
\EndFunction
		\end{algorithmic}

\end{algorithm} 

\subsection{Error Bounds}
Next, we more formally characterize the tightness of our proposed family of remainder-form decomposition functions in \eqref{eq:decomp1}, 
where we use the metric/measure of tightness in \eqref{eq:metric} that is based on the 
Hausdorff distance function. 
In particular, we 
derive lower and upper bounds on the \emph{over-approximation error} of the image set/range of function $f(\cdot)$, where the lower bound is achievable by our tight and tractable decomposition functions in Theorem \ref{thm:tractableDF}. 
\begin{thm}[Error Bounds] \label{thm:bounds}
Suppose that all the assumptions in Theorem \ref{thm:fdf} 
are satisfied for each $f_i 
 : \mathcal{Z} 
 \subset \mathbb{R}^{n_z} \to \mathbb{R}, i \in \mathbb{N}_{n_x}$. Let $T^f_O(\mathbb{I}\mathcal{Z}) \triangleq [\underline{f}^{\text{true}} ,\overline{f}^{\text{true}} ] \mo{\triangleq [ \min\limits_{z \in \mathbb{I}\mathcal{Z}} f(z),  \max\limits_{z \in \mathbb{I}\mathcal{Z}} f(z)]}=T^{f_d}_O(\mathbb{I}\mathcal{Z})$ be the tightest enclosing interval of the true image set/range of $f(\cdot)$ over $\mathbb{I}\mathcal{Z}=[\underline{z},\overline{z}] \in \mathbb{IR}^{n_z}$ (cf. Definition \ref{def:inc_func}), 
 $T_{\underline{\mathbf{m}},\overline{\mathbf{m}}}^{f_{d\mo{,i}}}(\mathbb{I}\mathcal{Z})\supseteq T^{f_\mo{i}}_O(\mathbb{I}\mathcal{Z})$, $\mo{\forall i \in \mathbb{N}_{n_x},}$ \mo{given in \eqref{eq:TR}}, be inclusion functions 
 using the family of decomposition functions in \eqref{eq:decomp1}, 
 \syo{and} \mo{$T^{f_d}_{\mathbb{M}}(\mathbb{I}\mathcal{Z}) \supseteq T^f_O(\mathbb{I}\mathcal{Z})$ be any inclusion function such that \syo{$T^{f_{d,i}}_{\mathbb{M}}(\mathbb{I}\mathcal{Z})=T_{\underline{\mathbf{m}},\overline{\mathbf{m}}}^{f_{d,i}}(\mathbb{I}\mathcal{Z})$,} $\forall i \in \mathbb{N}_{n_x}$, \syo{for some} 
 $\underline{\mathbf{m}},\overline{\mathbf{m}} \in \mathbf{M}_i$.} 
Then, the following 
inequalities hold:
\begin{enumerate}[(i)]
\item $\underline{q}_{f_d}(\mathbb{I}\mathcal{Z}) \leq  q( T_{\mo{\mathbb{M}}}^{f_d}(\mathbb{I}\mathcal{Z}) ,T^f_O(\mathbb{I}\mathcal{Z}))$, 
and
\item $\underline{q}_{f_d}(\mathbb{I}\mathcal{Z}) \leq \overline{q}_{f_d}(\mathbb{I}\mathcal{Z})  \leq  \hat{\overline{q}}_{f_d}(\mathbb{I}\mathcal{Z})$,
\end{enumerate}
 with the tightness metric $q([\underline{v},\overline{v}],[\underline{w},\overline{w}])
=\max\limits_{i\in \mathbb{N}_{n_x}} \max \{|\overline{v}_i-\overline{w}_i|,|\underline{v}_i-\underline{w}_i|\}$  
defined in \eqref{eq:metric}, 
and with the bounds $\underline{q}_{f_d}(\mathbb{I}\mathcal{Z})\triangleq \max_{i \in \mathbb{N}_{n_x}} \underline{q}_{f_{d,i}}(\mathbb{I}\mathcal{Z})$, $\overline{q}_{f_d}(\mathbb{I}\mathcal{Z})\triangleq \max_{i \in \mathbb{N}_{n_x}} \overline{q}_{f_{d,i}}(\mathbb{I}\mathcal{Z})$ and $\hat{\overline{q}}_{f_d}(\mathbb{I}\mathcal{Z})\triangleq \max_{i \in \mathbb{N}_{n_x}} \hat{\overline{q}}_{f_{d,i}}(\mathbb{I}\mathcal{Z})$, where 
for each $i \in \mathbb{N}_{n_x}$,
\begin{align*}
\begin{array}{rl}
\underline{q}_{f_{d,i}}(\mathbb{I}\mathcal{Z}) \hspace{-0.155cm}&\triangleq \max \{\min\limits_{\mathbf{m} \in \mathbf{M}^c_i}\Delta^1_{i,\mathbf{m}} \hspace{-.1cm}-\hspace{-.1cm}\overline{f}_i^{\text{true}},\underline{f}_i^{\text{true}}\hspace{-.1cm}-\hspace{-.1cm}\max\limits_{\mathbf{m} \in \mathbf{M}^c_i}\Delta^2_{i,\mathbf{m}} \},
\\
\hat{\overline{q}}_{f_{d,i}}(\mathbb{I}\mathcal{Z}) \hspace{-0.15cm}& \triangleq  \min\limits_{\mathbf{m} \in \mathbf{M}_i^c} \Delta^3_{i,\mathbf{m}}, \\
{\overline{q}}_{f_{d,i}}(\mathbb{I}\mathcal{Z}) \hspace{-0.15cm}& \triangleq  \min \{ \min\limits_{\mathbf{m} \in \mathbf{M}_i^c} \Delta^3_{i,\mathbf{m}} , \min\limits_{\mathbf{m} \in \mathbf{M}_i^c} \Delta^3_{i,\mathbf{m}}+ \Delta^4_{i,\mathbf{m}}\},
\end{array}
\end{align*}
with 
$\mathbf{M}_i$ 
and $\mathbf{M}_i^c$ 
 in \eqref{eq:Mj}, \eqref{eq:Mc} or \syo{\eqref{eq:Mj_CT}, \eqref{eq:Mc_CT}}, respectively, 
\begin{align}
\hspace{-0.3cm}\begin{array}{ll}
\Delta^1_{i,\mathbf{m}} \triangleq f_i( \zeta_{i,\mathbf{m}}^+)\hspace{-.cm}+\hspace{-.cm}\Delta^3_{i,\mathbf{m}} , \hspace{-0.25cm}& \Delta^2_{i,\mathbf{m}} \triangleq f_i(\zeta_{i,\mathbf{m}}^-)\hspace{-.cm}-\hspace{-.cm}\Delta^3_{i,\mathbf{m}}, \\
\Delta^3_{i,\mathbf{m}} \triangleq h_i(\zeta_{i,\mathbf{m}}^-)\hspace{-.05cm}-\hspace{-.05cm}h_i( \zeta_{i,\mathbf{m}}^+), 
 \hspace{-0.25cm}&  \Delta^4_{i,\mathbf{m}} \triangleq f_i( \zeta_{i,\mathbf{m}}^+)\hspace{-.05cm}-\hspace{-.05cm}f_i(\zeta_{i,\mathbf{m}}^-),
\end{array}\hspace{-0.35cm} \label{eq:Delta}
\end{align}
as well as $\zeta_{i,\mathbf{m}}^+ \triangleq \zeta_{i,\mathbf{m}}(\overline{z},\underline{z})$, $\zeta_{i,\mathbf{m}}^-\triangleq \zeta_{i,\mathbf{m}}(\underline{z},\overline{z})$ and  $\zeta_{i,\mathbf{m}}(.,.)$ defined in Theorem \ref{thm:fdf}. Further, without loss of tightness (cf. Lemma \ref{lem:mm_dec}), we can replace $\Delta^3_{i,\mathbf{m}}$ in \eqref{eq:Delta} with 
\begin{align} \label{eq:Delta3}
\Delta^3_{i,\mathbf{m}} \triangleq \mathbf{m}^\top(\zeta_{i,\mathbf{m}}^-\hspace{-.05cm}-\hspace{-.05cm} \zeta_{i,\mathbf{m}}^+),
\end{align}  
and the lower bound $\underline{q}_{f_d}(\mathbb{I}\mathcal{Z})$ is attained by the upper and lower decomposition functions (i.e., $T_R^{f_d}$) in Theorem \ref{thm:tractableDF}.
\end{thm}
\begin{proof}
First, from \eqref{eq:decomp1}, \eqref{eq:TR} 
and \eqref{eq:Delta}, we find
\begin{align} \label{eq:true_error}
\nonumber &\tilde{q}(T_{\underline{\mathbf{m}},\overline{\mathbf{m}}}^{f_{d,i}}(\mathbb{I}\mathcal{Z}),T^{f_i}_O(\mathbb{I}\mathcal{Z}))=\max \{\Delta^1_{i,\mathbf{m}} \hspace{-.1cm}-\hspace{-.1cm}\overline{f}_i^{\text{true}},\underline{f}_i^{\text{true}}\hspace{-.1cm}-\hspace{-.1cm}\Delta^2_{i,\mathbf{m}} \} \\
 &\geq \max \{\min_{\mathbf{m} \in \mathbf{M}^c_i}\Delta^1_{i,\mathbf{m}} \hspace{-.1cm}-\hspace{-.1cm}\overline{f}_i^{\text{true}},\underline{f}_i^{\text{true}}\hspace{-.1cm}-\hspace{-.1cm}\max_{\mathbf{m} \in \mathbf{M}^c_i}\Delta^2_{i,\mathbf{m}} \},
\end{align}
%
 where the inequality in (i) 
 follows from independently searching over $\mo{\mathbf{M}^c_i \subset \mathbf{M}_i}$ to minimize each argument of the maximization, \mo{as well as the fact that} by Lemmas \ref{lem:upper_lowwer_df} and \ref{lem:mm_dec}, we can apply \eqref{eq:Delta3} and search only over $\mathbf{M}^c_i \subset \mo{\mathbf{M}_i}$ without  any conservatism, and it can be verified that by construction, the lower bound $\underline{q}_f(\mathbb{I}\mathcal{Z})$ is attained by the decomposition functions in Theorem \ref{thm:tractableDF}. 

To obtain (ii), 
we apply 
\cite[Theorem 4-(b)]{cornelius1984computing}, which proved that for any remainder-form inclusion functions with remainder function $r_i(\cdot)$ satisfying $r_i(\mathbb{I}\mathcal{Z})\subset [\underline{r}_i,\overline{r}_i]$,
$\tilde{q}(W^R_{f_{d,i}}(\mathbb{I}\mathcal{Z}),V_{f_i}(\mathbb{I}\mathcal{Z}))\le \overline{r}_i-\underline{r}_i$ holds. 
For the first inequality \syo{in (ii)}, only $h_i(\cdot)$ 
is considered as the remainder function, while in the second inequality, both $h_i(\cdot)$ and $g_i(\cdot) \triangleq f_i(\cdot)-h_i(\cdot)$ 
are considered as remainder functions, separately, with the minimum chosen as the bound. Moreover, since $h_i(\cdot)$ is CJSS, by \eqref{eq:corner}, $\overline{h}_i$ and $\underline{h}_i$ are attained at the corner points given by $\zeta^-_{i,\mathbf{m}}$ and $\zeta^+_{i,\mathbf{m}}$, respectively, with $\Delta^3_{i,\mathbf{m}}\triangleq\overline{h}_i-\underline{h}_i$. Further, since $g_i(\cdot)$ is aligned with $-h_i(\cdot)$ by Lemma \ref{lem:gh_aligned}, $\overline{g}_i$ and $\underline{g}_i$ are attained at the corner points given by $\zeta^+_{i,\mathbf{m}}$ and $\zeta^-_{i,\mathbf{m}}$, respectively, with $\Delta^3_{i,\mathbf{m}}+\Delta^4_{i,\mathbf{m}}\triangleq\overline{g}_i-\underline{g}_i$. 
\end{proof}

The above result holds for both discrete-time and continuous-time systems (with overloading described in Definition \ref{def:inc_func}). 
Further, note that lower bound $\underline{q}_{f,d}(\mathbb{I}\mathcal{Z})$ is attainable by $T^{f_d}_R$ but since 
it is a function of the unknown $\overline{f}^{true}$ and $\underline{f}^{true}$, it cannot be computed. Thus, its upper bounds ${\overline{q}}_{f_d}(\mathbb{I}\mathcal{Z})$ and $\hat{\overline{q}}_{f_d}(\mathbb{I}\mathcal{Z}) $ in (ii) that are independent of $\overline{f}^{true}$ and $\underline{f}^{true}$ are more useful, e.g.,  as worst case function over-approximation error bounds in reachability and robust control problems.
\subsection{Convergence Rate and Subdivision Principle} 
In this subsection, 
we study the convergence rate of our proposed $T^{f_d}_R$, i.e., the rate at which its approximation error goes to zero, when the domain interval 
diameter $d(\mathbb{I}\mathcal{Z})$
shrinks. We show that when using $T^{f_d}_R$, the error converges at least linearly, which is also the convergence rate of natural inclusions $T^{f}_N$ \cite[Chapter 6]{moore2009introduction}. Further, we show that the \emph{subdivision principle}  introduced in \cite{alefeld2000interval} can be applied 
to improve the convergence. 
We first introduce the notion of convergence rate, inspired by 
\cite{alefeld2000interval}, and then, we present the convergence rate and the subdivision principle for our proposed $T^{f_d}_R$.  
\begin{defn}[Convergence Rate]
An inclusion function $T^f:\mathbb{IR}^{n_z} \to \mathbb{IR}^{n_x}$ for an ELLC vector field $f: \mathcal{Z} \subset \mathbb{R}_{n_z} 
\to \mathbb{R}^{n_x}$ has a convergence rate 
$\alpha >0$, if 
\begin{align}
q(T^f(\mathbb{I}\mathcal{Z}),T^f_O(\mathbb{I}\mathcal{Z}))\le \beta \, d(\mathbb{I}\mathcal{Z})^\alpha, 
\end{align}
\syo{for some $\beta >0$,} where $T\syo{^f}(\mathbb{I}\mathcal{Z})$ is the interval over-approximation of the range of $f(\cdot)$ over $\mathbb{I}\mathcal{Z}$ (i.e., an inclusion function), $T^f_O(\mathbb{I}\mathcal{Z})$ is the tightest inclusion function (cf. Definition \ref{def:inc_func}), 
$q(\cdot,\cdot)$ is defined in \eqref{eq:metric} 
and $d(\mathbb{I}\mathcal{Z}) \triangleq \|\overline{z}-\underline{z}\|_\infty$. 
\end{defn} 
\begin{thm}[Convergence Rate and Subdivision Principle for $T^{f_d}_R$]\label{thm:convergence_rate}
 The $T^{f_d}_R$ inclusion function for any ELLC vector field $f: \mathcal{Z} \subset \mathbb{R}_{n_z} 
\to \mathbb{R}^{n_x}$ satisfies: 
 \begin{align} \label{eq:conv_rate}
\syo{\underline{q}_{f_d}}(\mathbb{I}\mathcal{Z})=q(T_R^{f_d}(\mathbb{I}\mathcal{Z}),T^f_O(\mathbb{I}\mathcal{Z}))\le \beta^f_R d(\mathbb{I}\mathcal{Z}),
\end{align}
with convergence rate $\alpha=1$ for some $\beta^f_R>0$. Moreover, 
applying the subdivision principle, we have
 \begin{align} \label{eq:conv_subdiv}
q(T_R^{f_d}(\mathbb{I}\mathcal{Z};k),T^f_O(\mathbb{I}\mathcal{Z};k))\le \frac{\gamma^f_R}{k},
\end{align}
where $\mathbb{I}\mathcal{Z}$ is subdivided into $k^{n_z}$ interval vectors $\mathbb{I}\mathcal{Z}^l, l\in\mathbb{N}_{k^{n_z}}$ (i.e., with $k$ divisions in each dimension such that $d(\mathbb{I}\mathcal{Z}^l_j)=\frac{d(\mathbb{I}\mathcal{Z}_j)}{k}$ for $j\in\mathbb{N}_{n_z}$, $l\in\mathbb{N}_{k^{n_z}}$), $T_R^{f_d}(\mathbb{I}\mathcal{Z};k) \triangleq \bigcup_{l=1}^{k^{n_z}} T_R^{f_d}(\mathbb{I}\mathcal{Z}^l)$ and $T^f_O(\mathbb{I}\mathcal{Z};k) \triangleq \bigcup_{l=1}^{k^{n_z}} T^f_O(\mathbb{I}\mathcal{Z}^l)$.
\end{thm}
\begin{proof}
 For a linear remainder function $h(z)=\mathbf{m}^\top z$  with $\mathbf{m}\triangleq\begin{bmatrix}{\mathbf{m}}_1 \ \hdots \ {\mathbf{m}}_{n_x}\end{bmatrix}$ (cf. Lemma \ref{lem:mm_dec}), $\hat{\overline{q}}_{f_d}(\mathbb{I}\mathcal{Z})$ in Theorem  \ref{thm:bounds} can be upper bounded by triangle inequality by \eqref{eq:conv_rate} with
  $\beta^f_R=\max_{i \in \mathbb{N}_{n_x}} \|\mathbf{m}_i\|_{\infty}$.
  Further, the proof of \eqref{eq:conv_subdiv} follows the same 
  lines as the proof of \cite[Theorem 4.1]{ratschek1984computer}. 
\end{proof}
\subsection{Set Inversion Algorithm}
The remainder-form decomposition functions returned by Algorithm \ref{algorithm1} can be used with the generalized embedding system in Definition \ref{def:embedding} 
to over-approximate (unconstrained) reachable sets of a dynamic system governed by the vector field $f(\cdot)$, which corresponds to the propagated/predicted sets in state observers/estimators. However, when additional state constraint information is available (e.g., sensor observations/measurements in state estimation problems, known safety constraints from system design and manufactured constraints from modeling redundancy \cite{yang2020accurate,alamo2005guaranteed,alamo2008set,scott2013bounds,shen2017rapid}), an additional set inversion (also known as update or refinement) step will allow us to take the advantage of the constraints to shrink 
the propagated sets, i.e., to obtain a tighter subset of the propagated set that is compatible/consistent with the given constraints. Further details about the application of the reachable/propagated set and set inversion algorithms will be described in Section \ref{sec:application}.

Formally, given the constraint/observation function $\mu(\mo{x},u)$ in \eqref{eq:main_sys} with known $u$, a constraint/observation set $\mathcal{Y}$ with maximal and minimal values $\overline{y},\underline{y}$ (satisfying $\mu(\mo{x},u) \subset \mathcal{Y} \subseteq [\underline{y},\overline{y}]$) and a prior (or propagated/predicted) reachable interval $\mathbb{I}\mathcal{\mo{X}}^p=[\underline{\mo{x}}^p,\overline{\mo{x}}^p]$, we wish to find an updated/refined interval $\mathbb{I}\mathcal{\mo{X}}^u \subseteq \mathbb{I}\mathcal{\mo{X}}^p $, such that \eqref{eq:set_inv} holds, i.e., to solve Problem \ref{prob:set-inv}. 
 Finding $\mathbb{I}\mathcal{\mo{X}}^u$ in \eqref{eq:set_inv} is called the \emph{set inversion} problem \cite{jaulinapplied}. To our best knowledge, existing set inversion algorithms/operators either compute subpavings (i.e., unions of intervals) instead of an interval using (conservative) natural inclusions (SIVIA \cite[Chapter 3]{jaulinapplied}) or only applies if relatively restrictive monotonicity assumptions hold ($\mathcal{I}_G$ \cite[Algorithm 1]{yang2020accurate}). 
 
 In this section, leveraging our proposed \emph{discrete-time} decomposition-based inclusion functions for an ELLC function ${\nu}(\mo{x})\triangleq\mu(\mo{x},u)$ with known $u$, i.e., $T^{{\nu}_d}_R$ , we develop a novel set inversion algorithm that solves Problem \ref{prob:set-inv}, which is summarized in Algorithm \ref{algorithm2}. 
The main idea behind Algorithm \ref{algorithm2} is based on the observation that a candidate interval $
\mathbb{I}\Xi \triangleq [\underline{\xi},\overline{\xi}]
\subseteq {\mathbb{I}\mathcal{\mo{X}}}^p$ that satisfies $T_R^{\nu_d}(\mathbb{I}\Xi)
\cap [\underline{y},\overline{y}] = \emptyset$ 
(i.e., if $\underline{\nu}_d(\underline{\xi},\overline{\xi})>\overline{y}$ \emph{or} if $\overline{\nu}_d(\overline{\xi},\underline{\xi})<\underline{y}$)
 is \emph{incompatible/inconsistent} with the set $ \{\mo{x} \in \mathbb{I}\mathcal{\mo{X}}^p\, | \, \underline{y} \leq \mu(\mo{x},u) \leq \overline{y} \} $ and 
 can be eliminated/ruled out from $\mathbb{I}\mathcal{\mo{X}}^p$ and \syo{thus,} shrinking $\mathbb{I}\mathcal{\mo{X}}^u$. 

Using this idea, 
%
starting from the prior/propagated interval and using bisection for each dimension, Algorithm \ref{algorithm2} shrinks the compatible interval from below and/or above if $\underline{\nu}_d$ 
is strictly greater than $\overline{y}$ or 
if $\overline{\nu}_\syo{d}$ 
is strictly smaller than $\underline{y}$ (cf. Lines 8 and 18). 
Repeating this procedure along with \syo{bisections with a threshold $\epsilon$,} 
the candidate intervals that are determined to be \emph{inconsistent} with the constraint/observation set are ruled out. Note that the ordering of the dimensions in the `for' loop on Line 3 may have an impact on the tightness of the returned interval $\mathbb{I}\mathcal{\mo{X}}^u=[\underline{\mo{x}}^u,\overline{\mo{x}}_u]$ and \mo{so}, it may be desirable to tailor the order to the problem at hand, to randomize the order or to repeat the algorithm with the previous $\mathbb{I}\mathcal{\mo{X}}^u$ as $\mathbb{I}\mathcal{\mo{X}}^p$ multiple times.

The following result shows that Algorithm \ref{algorithm2} returns $\mathbb{I}\mathcal{\mo{X}}^u=[\underline{\mo{x}}^u,\overline{\mo{x}}_u]$ that satisfies 
\eqref{eq:set_inv}, i.e., solves Problem \ref{prob:set-inv}.  
 \begin{algorithm}[t] \small
\caption{Set Inversion based on Discrete-Time $T^{\nu_d}_R$}\label{algorithm2}
\begin{algorithmic}[1]
\Function{Set-Inv}{$\nu(\cdot),\overline{J}_C^\nu,\underline{J}_C^\nu,\overline{\mo{x}}^p,\underline{\mo{x}}^p,\overline{y},\underline{y},\epsilon$}
		\State Initialize: $\overline{\mo{x}}^u \gets \overline{\mo{x}}^p,\underline{\mo{x}}^u \gets\underline{\mo{x}}^p$; 
				\For {$\mo{i}=1$ to $n_{\mo{x}}$} 
                 \State $\underline{\zeta} \gets \underline{\mo{x}}^u_{\mo{i}}$; $\overline{\zeta} \gets \overline{\mo{x}}^u_{\mo{i}}$;
		\While{$\overline{\zeta}-\underline{\zeta} >\epsilon$}
		\State $\zeta_m \gets \frac{1}{2}(\overline{\zeta}+\underline{\zeta})$; $\overline{\xi} \gets \overline{\mo{x}}^u$; $\underline{\xi} \gets \underline{\mo{x}}^u$; $\underline{\xi}_{\mo{i}}  \gets \zeta_m$;
		\State $(\overline{\nu}_d,\underline{\nu}_d) \gets T^{\nu_d}_{R}(\nu(\cdot),\overline{J}_C^\nu,\underline{J}_C^\nu,\overline{\xi},\underline{\xi})$; (Algorithm \ref{algorithm1})
		\If{$(\overline{\nu}_d < \underline{y}) \ \vee \ (\underline{\nu}_d > \overline{y})  $}
		\State $\overline{\zeta} \gets \zeta_m$; $\overline{\mo{x}}^u_{\mo{i}} \gets \overline{\zeta}$;
		\Else{}
		\State $\underline{\zeta} \gets \zeta_m$;
		\EndIf
		\EndWhile
		\State $\underline{\zeta} \gets \underline{\mo{x}}^u_{\mo{i}} $; $\overline{\zeta} \gets \overline{\mo{x}}^u_{\mo{i}} $;
		\While{$\overline{\zeta}-\underline{\zeta} >\epsilon$}
		\State $\zeta_m \gets \frac{1}{2}(\overline{\zeta}+\underline{\zeta})$; $\overline{\xi} \gets \overline{\mo{x}}^u$; $\underline{\xi} \gets \underline{\mo{x}}^u$; $\overline{\xi}_{\mo{i}}  \gets \zeta_m$;
		\State $(\overline{\nu}_d,\underline{\nu}_d) \gets T^{\nu_d}_{R}(\nu(\cdot),\overline{J}_C^\nu,\underline{J}_C^\nu,\overline{\xi},\underline{\xi})$; (Algorithm \ref{algorithm1})
		\If{$(\overline{\nu}_d < \underline{y}) \ \vee \ (\underline{\nu}_d > \overline{y})  $}
		\State $\underline{\zeta} \gets \zeta_m$; $\underline{\mo{x}}^u_{\mo{i}} \gets \underline{\zeta}$;
		\Else{}
		\State $\overline{\zeta} \gets \zeta_m$;
		\EndIf
		\EndWhile
		\EndFor
		 \State \Return {$\overline{\mo{x}}^u,\underline{\mo{x}}^u$};
\EndFunction
		\end{algorithmic}

\end{algorithm}     

 \begin{prop} \label{lem:set_inv}
 Suppose Assumptions \ref{assum:ELLC} and \ref{assum:jacobian-bounds} hold and consider an ELLC constraint/observation function $\nu : \mathcal{\mo{X}}  \subset \mathbb{R}^{n_{\mo{x}}} \to  \mathbb{R}^{n_{\mo{\mu}}}$, where $\nu(\mo{x})\triangleq \mu(\mo{x},u)$ with known $u$, a constraint/observation set $\mathcal{Y}_t\subseteq [\underline{y}_t,\overline{y}_t]$ and a prior/propagated interval $\mathbb{I}\mathcal{\mo{X}}^p \triangleq [\underline{\mo{x}}_p,\overline{\mo{x}}_p] \in \mathbb{IR}^{n_{\mo{x}}}$.
   Then, the updated/refined interval $\mathbb{I}\mathcal{\mo{X}}^u\triangleq[\underline{\mo{x}}^u,\overline{\mo{x}}^u]$  
   returned by Algorithm \ref{algorithm2} satisfies \eqref{eq:set_inv}. 
 \end{prop} 
 \begin{proof}
 Obviously, $\mathbb{I}\mathcal{\mo{X}}^u\subseteq{\mathbb{I}\mathcal{\mo{X}}^p}$ (i.e., $\underline{\mo{x}}^p \leq \underline{\mo{x}}^u$ and $\overline{\mo{x}}^p \geq \overline{\mo{x}}^u$) by initialization and construction (cf. Lines 2, 9 and 19). Further, we show that $\mathbb{I}\mathcal{\mo{X}}^u \supseteq \mathbb{I}\mathcal{\mo{X}}^* \triangleq \{\mo{x} \in \mathbb{I}\mathcal{\mo{X}}^p \,|\, \underline{y} \leq \nu(\mo{x}) \leq \overline{y}\}$. To use contradiction, suppose that it does not hold. Then, $\exists \zeta \in \mathbb{I}\mathcal{\mo{X}}^*$ such that $\zeta \notin \mathbb{I}\mathcal{\mo{X}}^u$, i.e., $\exists \mo{i} \in \mathbb{N}_{n_{\mo{x}}}$ such that $\zeta_{\mo{i}} > \overline{\mo{x}}^u_{\mo{i}}$ or $\zeta_{\mo{i}} < \underline{\mo{x}}^u_{\mo{i}}$. Without loss of generality, suppose the first case holds, i.e., $\zeta_{\mo{i}} > \overline{\mo{x}}^u_{\mo{i}}$ (the proof for 
$\zeta_{\mo{i}} < \underline{\mo{x}}^u_{\mo{i}}$ is similar). Then, $\zeta \in [\underline{\mo{x}}^m, \overline{\mo{x}}^p]$, where $\underline{\mo{x}}^m_{\mo{i}}>\overline{\mo{x}}^u_{\mo{i}}$ and $\underline{\mo{x}}^m_{i'}=\underline{\mo{x}}^p_{i'}, \forall i' \ne i$. Hence, 
\begin{align}\label{eq:cr}
\underline{\nu}_{d}^R(\underline{\mo{x}}^m,\overline{\mo{x}}^p) \leq {\nu}(\zeta) \leq \overline{\nu}_{d}^R(\overline{\mo{x}}^p,\underline{\mo{x}}^m), 
\end{align}
where $\overline{\nu}_{d}^R(\cdot,\cdot)$ and $\underline{\nu}_{d}^R(\cdot,\cdot)$ are the proposed upper and lower remainder-form  decomposition functions in Algorithm \ref{algorithm1}. On the other hand, note that $\mathcal{\mo{X}}^u \cap [\underline{\mo{x}}^m, \overline{\mo{x}}^p]=\emptyset$, hence the interval $[\underline{\mo{x}}^m, \overline{\mo{x}}^p]$ has been ``ruled out" by Algorithm \ref{algorithm2}. In other words, one of the ``or" conditions in line 8 of Algorithm \ref{algorithm2} must hold for this interval, i.e., 
$\overline{\nu}_{d}^R(\overline{\mo{x}}^p,\underline{\mo{x}}^m)<\underline{y} \ \lor \ \underline{\nu}_{d}^R(\underline{\mo{x}}^m,\overline{\mo{x}}^p)<\overline{y}$.
Combining this and \eqref{eq:cr}, we obtain $\nu(\zeta) < \underline{y} \lor {\mu}(\zeta) > \overline{y}$, which contradicts with $\zeta \in \mathbb{I}{\mathcal{\mo{X}}}^*$ (i.e., $\underline{y} \leq \nu(\zeta) \leq \overline{y}$). 
 \end{proof}
 
 It is noteworthy that our set inversion algorithm can also be used with any applicable inclusion functions (such as $T^f_N,T^f_C,T^f_M, T^{f_d}_L, T^{f_d}_O$) or the best of them (i.e., by independently computing the reachable sets of all inclusion functions and intersecting them; cf. Corollary \ref{cor:min_max_dec}) 
 in place of $T^{f_d}_R$ in Lines 7 and 17. On the other hand, the proposed  $T^{f_d}_R$ (as well as $T^{f_d}_L, T^{f_d}_O$) can also be directly used in place of or in combination with natural inclusions within SIVIA \cite[Chapter 3]{jaulinapplied} to obtain subpavings \syo{(i.e., a union of intervals)}.

\section{Comparison with Existing Inclusion Functions}
\subsection{Comparison with the $T^{f_d}_L$ Inclusion Function}\label{sec:TL}
In this subsection, we compare the performance of the proposed $T^{f_d}_{R}$ with $T^{f_d}_L$ (cf. Proposition \ref{prop:Liren_dec}) through the following Theorem \ref{thm:Liren_mm}. We show that the decomposition function $f^L_d$ introduced in \cite{yang2020accurate} and recapped in Proposition \ref{prop:Liren_dec}, belongs to the family of the remainder-form decomposition functions in \eqref{eq:decomp1} and hence, $T^{f_d}_L$ cannot be tighter than $T^{f_d}_{R}$, which is the tightest decomposition function that belongs to \eqref{eq:decomp1}.
\begin{thm}[$T^{f_d}_L$ vs $T^{f_d}_{R}$] \label{thm:Liren_mm}
Suppose all the assumptions in Theorem \ref{thm:fdf} hold. Then, the following statements are true.
\begin{enumerate}[(i)]
\item $f^L_d(\cdot,\cdot)$ belongs to the family of decomposition functions in \eqref{eq:decomp1}, \mo{i.e.,} for each $i \in \mathbb{N}_{n_x}$, a specific pair of $\mathbf{m}_i^L \in \mathbf{M}_i$ and $h_i^L(\cdot) \in \mathcal{H}_{\mathbf{M}_i}$ corresponds to the decomposition function $f^L_d(\cdot,\cdot)$ in \cite[Theorem 2]{yang2019sufficient} (cf. Proposition \ref{prop:Liren_dec}).
\item The optimal remainder-form decomposition function $T^f_{R}$ is always tighter than (at least as good as) the inclusion function $T^f_L$, induced by the decomposition function $f^L_d$.
\end{enumerate}
\end{thm}
\begin{proof}
To prove (i), consider a specific decomposition function from the family of remainder functions in \eqref{eq:decomp1} that is constructed, for each $i \in \mathbb{N}_{n_x}$, with a supporting vector $\mathbf{m}_i^L$ and a linear remainder function ${h}^L_i(\cdot)=\moh{\langle}{\mathbf{m}}^L_{i},\cdot\moh{\rangle}$ as follows: 
\begin{align}\label{eq:m_Lir}
\hspace{-.2cm}&(\mathbf{m}^L_{i})_j\hspace{-.1cm}\\[-0.125cm]
\nonumber\hspace{-.2cm}&=\hspace{-.1cm}\begin{cases} \hspace{-.05cm}\min(({\underline{J}^{f}_C})_{ij},0), &\hspace{-.3cm} \text{if} \,  |\hspace{-.05cm}\min(({\underline{J}^{f}_C})_{ij},0)| \hspace{-.1cm}\leq \hspace{-.05cm} |\hspace{-.05cm}\max(({\overline{J}^{f}_C})_{ij},0)|, \\[-0.025cm]
 \hspace{-.05cm}\max(({\overline{J}^{f}_C})_{ij},0), &\hspace{-.3cm}  \text{if} \,   |\hspace{-.05cm}\max(({\overline{J}^{f}_C})_{ij},0)| \hspace{-.1cm}<\hspace{-.05cm} |\hspace{-.05cm}\min(({\underline{J}^{f}_C})_{ij},0)|,\end{cases}\hspace{-.4cm} 
\end{align}
for all $j\in \mathbb{N}_{n_z}$. 
Clearly, $\mathbf{m}^L_i \in \mathbf{M}^c_i \subset \mathbf{M}_i$ by its definition. Furthermore, it is easy to observe that $\mathbf{m}^L_i $ can be rewritten as
\begin{align}\label{eq:m_Lirr}
\hspace{-.25cm}(\mathbf{m}^L_i)_j\hspace{-.1cm}=\hspace{-.1cm}\begin{cases} 0, & \hspace{-.15cm}\text{if} \ (a_{ij} \geq 0) \lor (b_{ij} \leq 0) \lor (j=i), \\
a_{ij}, &\hspace{-.15cm} \text{if} \ (a_{ij} < 0) \land (b_{ij} > 0) \land (|a_{ij} | \leq |b_{ij}|),
\\ 
 b_{ij} &\hspace{-.15cm} \text{if} \ (a_{ij} < 0) \land (b_{ij} > 0) \land (|b_{ij} | \leq |a_{ij}|), \end{cases}\hspace{-.4cm}
\end{align}
where $a_{ij} \triangleq ({\underline{J}^f_C}_{ij})$, $b_{ij} \triangleq ({\overline{J}^f_C})_{ij}$. Recall that $(a_{ij} \geq 0)$, $(a_{ij} < 0) \land (b_{ij} > 0) \land (|a_{ij} | \leq |b_{ij}|)$, $(a_{ij} < 0) \land (b_{ij} > 0) \land (|b_{ij} | \leq |a_{ij}|)$, $(b_{ij} \leq 0)$ and $j=i$ (continuous-time systems only) correspond to Cases $1$--$5$ in Proposition \ref{prop:Liren_dec}, respectively. 
Then, by 
\eqref{eq:m_Lir} and \eqref{eq:m_Lirr}, we find that $\zeta_{\mathbf{m}^L_i,j}(z,\hat{z})$ in \eqref{eq:corner} coincides  
with $\zeta_j$ in Proposition \ref{prop:Liren_dec}.
Moreover, by \eqref{eq:m_Lirr}, 
\begin{align}
\begin{array}{l}
{h}^L_i(\zeta_{\mathbf{m}^L_i}(\hat{z},{z}))=\langle \mathbf{m}^L_i,\zeta_{\mathbf{m}^L_i}(\hat{z},{z})\rangle =\sum_{j=1}^{\mathbb{N}_{n_z}} \phi^i_j, \\
{h}^L_i(\zeta_{\mathbf{m}^L_i}({z},\hat{z}))=\langle \mathbf{m}^L_i,\zeta_{\mathbf{m}^L_i}({z},\hat{z})\rangle =\sum_{j=1}^{\mathbb{N}_{n_z}}  \psi^i_j, 
 \end{array}
\end{align}
where \syo{$\phi^i_j \hspace{-0.1cm}\triangleq\hspace{-0.1cm} \begin{cases}  0, &\hspace{-0.275cm} \text{Cases} \, 1,\hspace{-0.05cm}4,\hspace{-0.05cm}5, \\ a_{ij} \hat{z}_{j}, &\hspace{-0.275cm} \text{Case} \ 2, \\ b_{ij}z_{j}, &\hspace{-0.275cm} \text{Case} \ 3, \end{cases}$\hspace{-0.15cm} and $\psi^i_j \hspace{-0.1cm}\triangleq\hspace{-0.1cm} \begin{cases}  0, &\hspace{-0.275cm} \text{Cases} \, 1,\hspace{-0.05cm}4,\hspace{-0.05cm}5, \\ a_{ij} {z}_{j}, &\hspace{-0.275cm} \text{Case} \ 2, \\ b_{ij}\hat{z}_{j}, &\hspace{-0.275cm} \text{Case} \ 3. \end{cases}\hspace{-0.1cm}$ 
Consequently,} 
\begin{align*}
{h}^L_i(\zeta_{\mathbf{m}^L_i}(\hat{z},{z})) \hspace{-0.05cm}-\hspace{-0.05cm} {h}^L_i(\zeta_{\mathbf{m}^L_i}({z},\hat{z}))=\textstyle\sum_{j=1}^{\mathbb{N}_{n_z}}  \phi^i_j \hspace{-0.05cm}-\hspace{-0.05cm}\psi^i_j =\sum_{j=1}^{\mathbb{N}_{n_z}} \theta^i_j,
\end{align*}
where $\theta^i_j= \begin{cases}  0, & \text{Cases} \, 1,4,5, \\ a_{ij} (\hat{z}_{j}-z_{j}), & \text{Case} \ 2, \\ b_{ij}(z_{j}-\hat{z}_{j}), & \text{Case} \ 3. \end{cases}$ Then, defining two indicator functions $\alpha^i,\beta^i \in \mathbb{R}^{n_z}$, where for all $j \in \mathbb{N}_{n_z}$, $\alpha^i_{j}\hspace{-0.05cm} \triangleq\hspace{-0.05cm} \begin{cases} 0, &\hspace{-0.15cm} \text{Cases} \, 1,3,4,5, \\ |a_{ij}|, &\hspace{-0.15cm} \text{Case} \ 2,\end{cases}$  and $\beta^i_{j}  \triangleq  \begin{cases} 0, &\hspace{-0.15cm} \text{Cases} \, 1,2,4,5, \\ -|b_{ij}|, &\hspace{-0.15cm} \text{Case} \ 3,\end{cases}$  $\theta^i_j$ can be rewritten as $\theta^i_j=(\alpha^i_j -\beta^i_j)(z_j-\hat{z}_j)$, and hence
\begin{align*}
{h}^L_i(\zeta_{\mathbf{m}^L_i}(\hat{z},{z})) \hspace{-0.05cm}-\hspace{-0.05cm} {h}^L_i(\zeta_{\mathbf{m}^L_i}({z},\hat{z}))=\textstyle\sum_{j=1}^{\mathbb{N}_{n_z}}  \theta^i_j=\langle \alpha^i \hspace{-0.05cm}-\hspace{-0.05cm} \beta^i,z \hspace{-0.05cm}-\hspace{-0.05cm} \hat{z}\rangle.
\end{align*} 
Finally, since $\zeta_{\mathbf{m}^L_i}({z},\hat{z})$ coincides with $\zeta$ in Proposition \ref{prop:Liren_dec}, by \eqref{eq:decomp1}, 
$f^i_d(z,\hat{z};\mathbf{m}^L_i,h^L_i(\cdot))=h^L_i(\zeta_{\mathbf{m}^L_i}(\hat{z},{z}))+f_i(\zeta_{\mathbf{m}^L_i}({z},\hat{z})) -h_i(\zeta_{\mathbf{m}^L_i}({z},\hat{z})))=f_i(\zeta)+\langle \alpha^i-\beta^i,z-\hat{z}\rangle=f^L_{d,i}(z,\hat{z})$, where $f^L_{d,i}(z,\hat{z})$ is the decomposition function introduced in Proposition \ref{prop:Liren_dec} and is defined in \eqref{eq:Lir_dec}.  

 
(ii) The result directly follows from (i) and Theorem \ref{thm:tractableDF}. 
\end{proof}

From the above, we know that $T^{f_d}_L$, which only considers a specific $\mathbf{m}_i\in \mathbf{M}_i$ for all $i\in \mathbb{N}_{n_x}$, cannot be tighter than $T^{f_d}_L$, which considers all $\mathbf{m}_i\in \mathbf{M}_i$. 
Nonetheless, since $T^{f_d}_L$ requires less computation, it can still be useful for 
system\syo{s with large} dimensions, and can also be tighter than $T^f_N,T^f_C$ and $T^f_M$ (see Examples \ref{ex:ex1} and \ref{ex:ex2} below). Further, this 
suggests that when computational resources are limited, it is also possible to consider a strict subset of $\mathbf{M}_i$ on top of the one in $T^{f_d}_L$ to obtain a tighter decomposition function than $T^{f_d}_L$. \syo{In addition, the above theorem indirectly proves that $T^{f_d}_L$ in Proposition \ref{prop:Liren_dec} also applies to ELLC systems.}
\subsection{Comparison with $T^f_N$, $T^f_C$ and $T^f_M$ Inclusion Functions}
In this subsection, we compare the performance of natural inclusions and some of their modifications, i.e., $T^f_N$, $T^f_C$, $T^f_M$ with the (discrete-time) $T^{f_d}_{R}$, via computing the over-approximation for the range of some example functions. It is worth mentioning that we were not able to derive any theoretical results that show the superiority of one over the others. In fact, our simulation results showed that depending on function and its corresponding considered domain, one of them can be tighter than the others in some cases and the opposite holds for other cases. However, in some cases, reflected in the following examples, the $T^{f_d}_R$ typically returns tighter intervals. 
\subsubsection{Composition of Non-Elementary Functions}
In cases where the considered vector field is not a composition of ``elementary functions" (e.g., simple monomials, $\sin(\cdot)$, $\cos(\cdot)$, monotone functions, etc\syo{.}), $T^f_N$, $T^f_C$ and $T^f_M$ are known to be hard 
to compute 
and conservative over-approximations for bounding the constituent functions are often needed, which lead to poor inclusion functions, i.e., large errors. In these cases, it is most likely that $T^{f_d}_{R}$ returns better bounds. The following example describes one such function.  
\begin{exm} \label{ex:ex1}
Consider 
$f(x)=x\arctan{(x^2-2x+5)}$, which is composed of non-elementary functions, and an interval domain $\mathbb{I}\mathcal{X}=[1,3]$. 
In this case, $T^f_N$, $T^f_C$, $T^f_M$, $T^{f_d}_{L}$ and $T^{f_d}_{R}$ return $[-4.7124,4.7124]$, $[1.3258,4.3393]$, $[1.3187,4.2475]$, $[1.2835,2.9461]$ and $[1.1760,2,7468]$, respectively, where the final interval (corresponding to $T^{f_d}_{R}$) is a subset of all others.
\end{exm}
\subsubsection{``Almost" Sign-Stable Functions}
In cases where $f(\cdot)$ can be decomposed into a CJSS constituent and a relatively small additive perturbation,  $T^f_{R}$ \syo{will most likely} return tighter bounds \moh{\syo{than} 
the bounds returned by} $T^f_N$, $T^f_C$ and $T^f_M$. For instance, consider the following example. 
\begin{exm}  \label{ex:ex2}
Consider 
$f(x)=x^3-0.1x$, which is a monotone increasing (and hence CJSS) function on its interval domain $\mathbb{I}\mathcal{X}=[-1,3]$, except on the short interval $[-\sqrt{\frac{0.1}{3}},\sqrt{\frac{0.1}{3}}]$. For this example, $T^f_N$, $T^f_C$, $T^f_M$, $T^{f_d}_{L}$ and $T^{f_d}_{R}$ return $[-8.9000,27.0100]$,$[-49.9000,54.7000]$,$[-49.9000,54.7000]$, $[-1.0300,26.0100]$ and $[-1.0300,26.0100]$, respectively, where $T^{f_d}_L$ and $T^{f_d}_R$ are much tighter than $T^f_N$, $T^f_C$ and $T^f_M$.  
\end{exm}
\subsubsection{Vector Fields with Several Additive Terms}
It is also well-known in the literature that natural, centered and mixed-centered inclusions perform worse for the functions with many additive terms, compared to the ones with fewer additive terms \cite{jaulinapplied,moore2009introduction}. This is not necessarily true for the performance of $T^{f_d}_{R}$. The following example illustrates this fact, where a function with several additive terms is considered.  
\begin{exm} \label{ex:ex3}
Consider $f(x)=x_1x_2x_3+x_1^2x_2+x_2^2x_3+x_3^2x_1+x_1^2x_3+x_3^2x_2+x_2^2x_1+x_1^3+x_2^3+x_3^3$ with an interval domain $\mathbb{I}\mathcal{X}=[-2,2]\times [-2,2] \times [-2,2]$. 
Then, $T^f_N$, $T^f_C$, $T^f_M$, $T^{f_d}_{L}$ and $T^{f_d}_{R}$ return $[-80,80]$, $[-76.45,76.45]$, $[-73.62,73.62]$, $[-176,176]$ and $[-54.4,54.4]$, respectively, where the final interval from $T^{f_d}_{R}$ is the tightest among all.
\end{exm}
\subsubsection{Existence of Closed-Form Decomposition Functions}
Finally, it is notable that our proposed $T^{f_d}_R$ approach (and $T^{f_d}_L$) enables us to find closed-form inclusion functions for a wide class of vector fields. This, can be analytically beneficial, e.g., in convergence analysis for reachable sets or stability analysis in interval observer designs \cite{khejenejad2020full,khejenejad2021def}. This is in contrast to natural, centered and mixed-centered inclusions (and also $T^{f_d}_O$ in general), where a closed-form inclusion function for general classes of functions, is often not available. 
\section{Applications} \label{sec:application}
\subsection{Application to Constrained Reachability Analysis}
Consider the following constrained bounded-error system:
\begin{gather} \label{eq:sys1}
x_t^+=f(x_t,u_t,w_t),\quad
\mu(x_t,u_t) \in \mathcal{Y}_t, 
\end{gather}
where $x_t^+ \triangleq x_{t+1}$ if \eqref{eq:sys1} is a discrete-time \sy{system} (with sampling time $\delta t$) and $x_t^+ \triangleq \dot{x}_t$ if \eqref{eq:sys1} is a continuous-time system, $x_t \in \mathbb{R}^{n_x}$ with $x_0 \in [\underline{x}_0,\overline{x}_0]$ and $u_t \in \mathbb{R}^{n_u}$ are state and known input signals, $w_t \in  [\underline{w},\overline{w}] \in \mathbb{I}\mathbb{R}^{n_w}$ 
is a bounded process disturbance signals, 
$\mathcal{Y}_t\subseteq[\underline{y}_t,\overline{y}_t] \in \mathbb{IR}^{n_y}$ is the time-varying, \emph{uncertain} state interval constraint and 
$f:\mathbb{R}^{n_x+n_u+n_w} \to \mathbb{R}^{n_x},\mu:\mathbb{R}^{n_x+n_u} \to \mathbb{R}^{n_{\mo{\mu}}}$ are known vector fields. 
The following proposition shows how to apply Algorithms \ref{algorithm1}--\ref{algorithm2}, i.e., the   mixed-monotone remainder-form decomposition function construction and the set inversion algorithms, to compute \syo{over-}approximations of the reachable sets\syo{/framers} of the states for the system in \eqref{eq:sys1}. 

\begin{prop}\label{prop:reach}
Consider the system \eqref{eq:sys1} with initial state $x_0 \in \mathbb{I}\mathcal{X}_0 \triangleq  [\underline{x}_0,\overline{x}_0]$ and let \syo{$f(z_t)\triangleq \tilde{f}(x_t,u_t,w_t)$} and $\nu(x_t)\triangleq \mu(x_t,u_t)$ with \mo{$z_t\triangleq[x_{\mo{t}}^\top \ w_t^\top]^\top$} and known $u_t$. Suppose that Assumptions \ref{assum:ELLC} and \ref{assum:jacobian-bounds} hold
 and $\epsilon$ is a chosen small positive threshold. Then, for all $t\ge0,CR^f(t,\mathbb{I}\mathcal{X}_0) \subset \mathbb{I}\mathcal{X}^u_t \triangleq [\underline{x}^u_t,\overline{x}^u_t]$, 
where $CR^f(t,\mathbb{I}\mathcal{X}_0) \triangleq
\{\phi(t, x_0, {w}_{\mo{t}}) \mid x_0 \in \mathbb{I}\mathcal{X}_0, \mu(x_t,u_t) \in \syo{\mathcal{Y}_t \subseteq [\underline{y}_t,\overline{y}_t]} \text{ and } w_t \in \mathbb{I}\mathcal{W},  \forall t\ge 0\}$ is the constrained reachable set at time $t$ of \eqref{eq:sys1} 
when initialized within $\mathbb{I}\mathcal{X}_0$ and $ \mathbb{I}\mathcal{X}^u_t \triangleq [\underline{x}^u_t,\overline{x}^u_t]$ is the solution to the following constrained embedding system:
\begin{align*}
\begin{bmatrix}{\overline{x}}_t^{p+} \\ {\underline{x}}_t^{p+} \end{bmatrix}&=\hspace{-0.1cm}\begin{bmatrix} \overline{\syo{f}}_d(\begin{bmatrix}(\overline{x}_t^u)^\top \, \overline{w}^\top\end{bmatrix}^\top\hspace{-0.05cm},\begin{bmatrix}(\underline{x}_t^u)^\top \, \underline{w}^\top \end{bmatrix}^\top) \\ \underline{\syo{f}}_d(\begin{bmatrix}(\underline{x}_t^u)^\top \, \underline{w}^\top \end{bmatrix}^\top\hspace{-0.05cm},\begin{bmatrix}(\overline{x}_t^u)^\top \, \overline{w}^\top\end{bmatrix}^\top) \end{bmatrix}\hspace{-0.05cm}, \ \begin{bmatrix}\overline{x}^p_0\\ \underline{x}^p_0\end{bmatrix}\hspace{-0.1cm}=\hspace{-0.1cm}\begin{bmatrix}\overline{x}_0\\ \underline{x}_0\end{bmatrix},\\
(\overline{x}^u_t,\underline{x}^u_t)&=\text{SET-INV}({\nu}(\cdot),\overline{J}^{{\nu}}_C,\underline{J}^{{\nu}}_C,\overline{x}^p_t,\underline{x}^p_t,\overline{y}_t,\underline{y}_t,\epsilon),
\end{align*}
where $(\overline{\syo{f}}_d(\cdot,\cdot),\underline{\syo{f}}_d(\cdot,\cdot))=T^{\syo{f}_d}_R(\syo{f}(\cdot),\overline{J}^{{\syo{f}}}_C,\underline{J}^{\syo{f}}_C,\cdot,\cdot)$ is the discrete-time or continuous-time decomposition-based inclusion function in Algorithm \ref{algorithm1} and the SET-INV function in Algorithm \ref{algorithm2} is based on the discrete-time $T^{\nu_d}_R(\nu(\cdot),\overline{J}^{{\nu}}_C,\underline{J}^{{\nu}}_C,\cdot,\cdot)$ from Algorithm \ref{algorithm1}. Consequently,
the constrained system state trajectory $x_t$ 
satisfies $\underline{x}^u_t \le x_t \le \overline{x}^u_t$ at all times $t$.
\end{prop}
\begin{proof}
The results directly follow from applying \sy{Propositions \ref{cor:embedding}--\ref{cor:dec_inc},} 
Theorems \ref{thm:fdf}--\ref{thm:tractableDF} and Lemmas \ref{lem:mm_dec}--\ref{lem:set_inv}. 
\end{proof}
\subsection{Application to Interval Observer Design}
Now, we consider the interval observer design problem for the following bounded-error system:
\begin{align} \label{eq:sys2}
x^+=f(x_t,u_t,w_t),\quad
y_t={\mu}(x_t,u_t) +Vv_t, 
\end{align}
where its state equation is similar to the system \eqref{eq:sys1}, but instead of a known state constraint, an observation/measurement signal $y_t \in \mathbb{R}^{n_\syo{{y}}}$ is known/measured with a 
known observation function ${\mu}(\cdot):\mathbb{R}^{n_x+n_u} \to \mathbb{R}^{n_{\syo{y}}} $, known matrix $V \in \mathbb{R}^{n_{\syo{y}} \times \syo{n}_v}$ and measurement noise signal $v_t \in [\underline{v},\overline{v}] \in \mathbb{IR}^{n_v}$. It can be easily verified that the observation equation can be equivalently written as ${\nu}(x_t) \triangleq {\mu}(x_t,u_t) \in \mathcal{Y}_t \triangleq [y_t-\overline{s},y_t-\underline{s}]$, where $\overline{s}=V^{+}\overline{v}-V^{-}\underline{v}$, $\underline{s}=V^{+}\underline{v}-V^{-}\overline{v}$, $V^{+} \triangleq \max(V,\mathbf{0}_{n_v})$ and $V^{-} \triangleq V^{+} - V$, with $\mathbf{0}_{n_v}$ being a zero vector in $\mathbb{R}^{n_v}$ \cite{efimov2013interval}. This transforms the system \eqref{eq:sys2} into the form of \eqref{eq:sys1} and hence, Proposition \ref{prop:reach} directly applies, where $[\underline{x}^p_t,\overline{x}^p_t]$ is the state interval from the prediction/propagation step and $[\underline{x}^u_t,\overline{x}^u_t]$ is the updated state interval after a measurement update step. Further, if \eqref{eq:sys2} is a sampled-data system, i.e., the system dynamics is continuous and the observations are sampled in discrete-time, our proposed approach still directly applies \syo{with minor modifications}.

\section{Simulations}
In this section, we compare the performance\syo{s} of $T^f_N$ (natural inclusions; cf. Proposition \ref{prop:natural}), $T^f_C,T^f_M$ (centered and mixed-centered inclusions; cf. Proposition \ref{prop:natural_inclusions}), $T^{f_d}_L$ (decomposition functions proposed in \cite{yang2019sufficient}; cf. Proposition \ref{prop:Liren_dec}), $T^{f_d}_{R}$ (our proposed remainder-form decomposition function in Algorithm \ref{algorithm1}), $T^f_{S_1}$ (the first proposed bounding approach in \cite[\syo{Theorem 1}]{yang2020accurate}, if applicable), $T^f_{S_2}$ (the second proposed approach in \cite[\syo{Theorem 2}]{yang2020accurate}, if applicable) and $T^{f_d}_{O}$ (the tight decomposition functions proposed in \cite[Theorem 2]{yang2019tight} for discrete-time and in \cite[Theorem 1]{abate2020tight} for continuous-time systems, when computable; cf. Proposition \ref{prop:tight_dec}) in computing the reachable sets of several unconstrained and constrained dynamical systems in the form of \eqref{eq:main_sys}. Further, by the intersection property in Corollary \ref{cor:min_max_dec}, we can also intersect the reachable sets of all applicable inclusion functions (except $T^{f_d}_{O}$) for comparison. Note that we only compare methods in the literature that use \emph{first-order} (generalized) gradient information. The consideration of higher-order information is a subject of our future work.
\subsection{Van Der Pol System}
Consider the following discretized Van der Pol system \cite{shen2017rapid}:
\begin{align} \label{eq:Van_der}
\begin{array}{ll}
x_{1,t+1}=x_{1,t}+\delta t \, x_{2,t},\\
x_{2,t+1}=x_{2,t}+\delta t \, ((1-x^2_{1,t})x_{2,t}-x_{1,t}),
\end{array}
\end{align}
with $\delta t =0.1$ s and $x_0\in[1.15,1.4] \times [2.05,2.3]$. 
Starting from the initial intervals, the results for computing reachable intervals, using several applicable inclusion functions, are depicted in \syo{Figure} \ref{fig:van_der_pol}. Unfortunately, $T^f_{S_1}$ and $T^f_{S_2}$ are not valid (applicable) here, due to the lack of required monotonicity assumptions (cf. \cite[conditions (6) and (16)]{yang2020accurate}), while as expected, $T^{f_d}_{O}$ (that is computable here,  since the corresponding optimization problems can be analytically solved by computing and comparing all stationary points of the vector fields and the boundaries) returns the tightest bounds. Moreover, our proposed $T^{f_d}_{R}$ notably returns tighter bounds than all applicable methods, except $T^{f_d}_{O}$ and the best of $T^f_N$--$T^{f_d}_{R}$, and seems to offer significant improvement over natural, centered and mixed-centered inclusions as well as $T^{f_d}_L$. Further, ``the best of $T^f_N$--$T^{f_d}_{R}$" approach is able to get very close to the tightest set obtained by $T^{f_d}_{O}$. 
\begin{figure}[t]
\begin{center}
\includegraphics[scale=0.225,trim=39mm 0mm 35mm 5mm,clip]{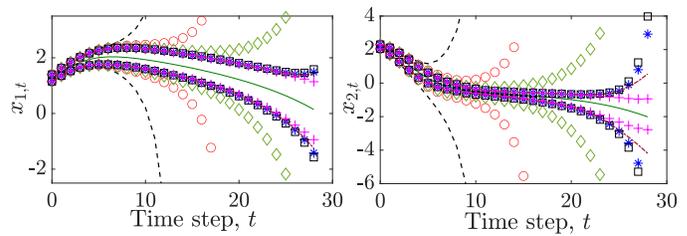}
\vspace{-0.1cm}
\caption{Upper and lower bounds on ${x}_{1}$ and $x_2$ in System \eqref{eq:Van_der} (Example A: Van der Pol system), when applying $T^f_N (- -)$, $T^f_C$ ({\color{red}$\circ$}), $T^f_M$ ({\color{ForestGreen}$\diamond$}), $T^{f_d}_L$ ($\square$), $T^{f_d}_{R}$ ({\color{blue}$\ast$}), the best of $T^f_N$--$T^{f_d}_{R}$ ({\color{purple}$\cdot$-}) and $T^{f_d}_{O}$ ({\color{magenta}+}), as well as the midpoint trajectory ({\color{ForestGreen}--}). \label{fig:van_der_pol}}
\end{center}
\vspace{-0.15cm}
\end{figure}

\subsection{Example 3 in \cite{yang2020accurate} with Modeling Redundancy}\label{ex:redundancy}
\begin{figure}[t]
\begin{center}
\vspace{-0.1cm}
\begin{subfigure}[b]{0.475\textwidth}
\includegraphics[scale=0.221,trim=43mm 0mm 35mm 10mm,clip]{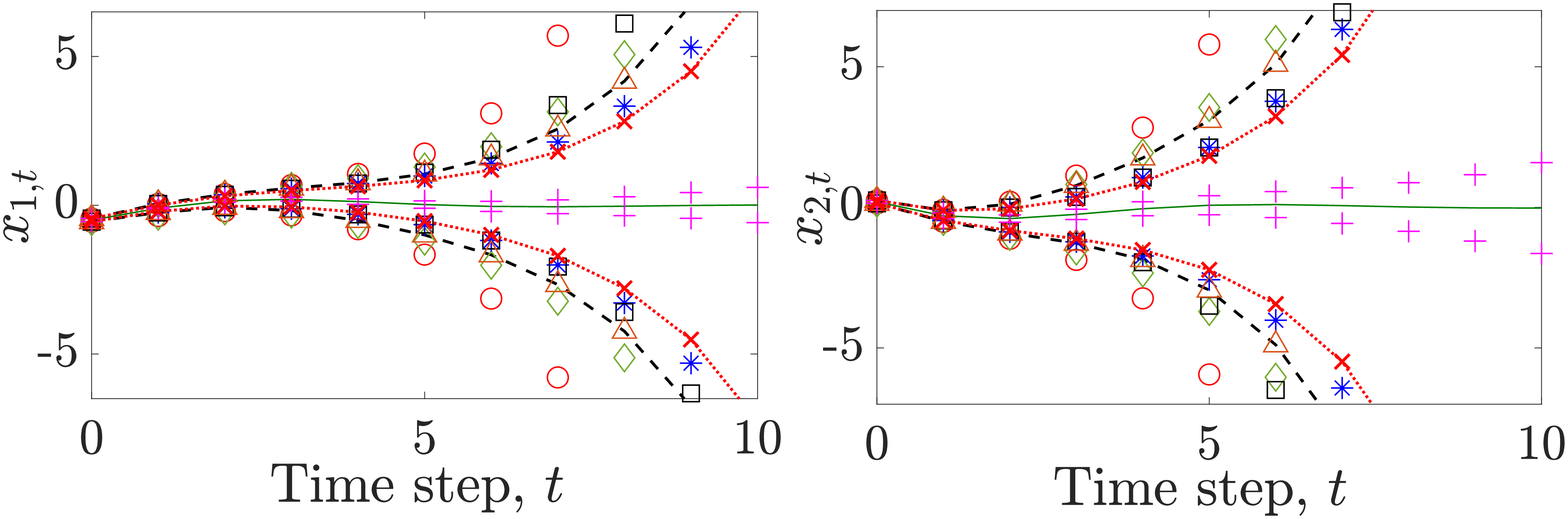}
\vspace{-0.1cm}
\caption{System B without observations \eqref{eq:ex3_SC}--\eqref{eq:aug_Scott_ex3}. 
\label{fig:ex3_scott_1}}
\end{subfigure}
%
\begin{subfigure}[b]{0.475\textwidth}
\includegraphics[scale=0.221,trim=45mm 0mm 35mm 0mm,clip]{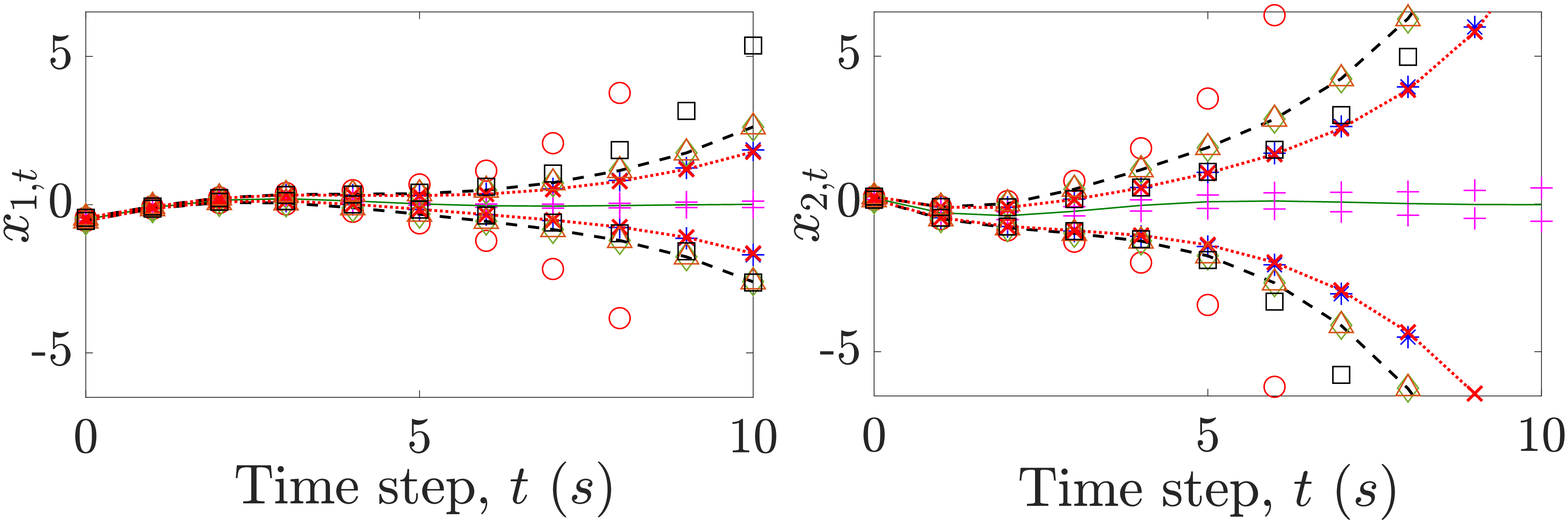}
\vspace{-0.1cm}
\caption{System B with observations \eqref{eq:ex3_SC}--\eqref{eq:ex3_SC_2}. 
\label{fig:ex3_scott_2}}
\end{subfigure}
\caption{Upper and lower bounds/framers for ${x}_{1}$ and $x_2$ in Example B, 
when applying $T^f_N (- -)$, $T^f_{S_1}$ ({\color{brown}$\triangle$}), $T^f_C$ ({\color{red}$\circ$}), $T^f_M$ ({\color{ForestGreen}$\diamond$}), $T^{f_d}_L$ ($\square$), $T^{f_d}_{R}$ ({\color{blue}$\ast$}), $T^{f_d}_{R}$ with manufactured redundancies ({\color{red}$\times\cdot$}) and $T^{f_d}_{O}$ ({\color{magenta}+}), as well as the midpoint trajectory ({\color{ForestGreen}--}). \label{fig:ex3_scott_all}}
\end{center}
\vspace{-0.45cm}
\end{figure}
Now consider the following discrete-time dynamical system from \cite{alamo2005guaranteed,yang2020accurate} with bounded noise:
\begin{align} \label{eq:ex3_SC}
x_{t+1}=\begin{bmatrix} 0 & -0.5 \\ 1 &1+0.3\mo{\lambda}_t \end{bmatrix} x_t+0.02\begin{bmatrix} -6 \\ 1 \end{bmatrix} w_t,
\end{align}
where $w_t$, $\mo{\lambda}_t \in [-0.001,0.001]$ and $x_0 \in [-0.55,-0.445] \times [0.145,0.248]$. The approximated reachable sets are depicted in \syo{Figure} \ref{fig:ex3_scott_1}. Here, $T_{S_1}$ is applicable and returns the exact same bounds as $T^f_N$ (natural inclusion), but $T^f_{S_2}$ is not applicable due to the lack of monotonicity (cf. \cite[\syo{(16)}
]{yang2020accurate}). The tight $T^{f_d}_{O}$ is again computable and, as expected, returns the tightest possible intervals. Again, our proposed approach, $T^{f_d}_R$ outperforms all applicable ones, except for the tight $T^{f_\syo{d}}_{O}$. To further improve our results and inspired by the literature on modeling redundancy, e.g., \cite{scott2013bounds,yang2020accurate}, we consider a manufactured \emph{redundant} state $z_t$ chosen as: 
\begin{align}
 \label{eq:const_ex3_Sc} z_t&=x_{1,t}+6x_{2,t}, 
\end{align}
which implies that:
 \begin{align} 
 z_{t+1}&=z_t+5x_{1,t}+(1.8v_t-0.5)x_{2,t}. \label{eq:aug_Scott_ex3} 
 \end{align} 
 Augmenting \eqref{eq:aug_Scott_ex3} with the original system \eqref{eq:ex3_SC}, we consider the over-approximation of the reachable sets for the augmented system, subject to \eqref{eq:const_ex3_Sc} as a constraint, using our proposed $T^{f_d}_{R}$ for the propagation step, as well as Algorithm \ref{algorithm2} for set inversion (refinement/update). The results in \syo{Figure} \ref{fig:ex3_scott_1} show an improvement when considering the manufactured redundant variable \syo{when} compared to 
 using $T^{f_d}_{R}$ without any redundancies. It can also be shown that different manufactured states result in different amounts of improvement. Thus, as future work, we will also consider the development of a \emph{principled} approach for choosing manufactured/redundant variables. 

Next, we consider \eqref{eq:ex3_SC}--\eqref{eq:aug_Scott_ex3} with an observation equation:
 \begin{align} \label{eq:ex3_SC_2}
y_t=1.6x_{1,t}+0.3x_{2,t}+v_t,
\end{align}
where $y_t$ is a \emph{known} measurement/observation signal at time step $t$ with measurement noise signal $v_t \in [-.05,0.05]$. \syo{Figure} \ref{fig:ex3_scott_2} shows the approximated upper and lower bounds for the states of the system \eqref{eq:ex3_SC_2}, using the same inclusion functions that we applied to the system \eqref{eq:ex3_SC}, as well as applying Algorithm \ref{algorithm2} for the set inversion/refinement procedure. As expected and as can be seen by comparing Figures \ref{fig:ex3_scott_1} and \ref{fig:ex3_scott_2}, the additional measurement information enables us to tighten the resulting intervals for all inclusion functions.   

\subsection{Example 2.11 in \cite{jaulinapplied}}
Next, we consider the following discrete-time dynamical system in \cite[Example 2.11]{jaulinapplied}: 
\begin{align} \label{eq:Jaulin_ex_2_11}
\begin{array}{ll}
x_{1,t+1}=x_{1,t}^2+x_{1,t}e^{x_{2,t}}-x_{2,t}^2,\\
x_{2,t+1}=x_{1,t}^2-x_{1,t}e^{x_{2,t}}+x_{2,t}^2, 
\end{array}
\end{align}
where $x_{0} \in [0.12,0.121] \times [0.182,0.185]$. Here, $T^f_{S_1}$ and $T^f_{S_2}$ approaches are not applicable, due to lack of the required monotonicity conditions, while the tight $T^{f_d}_{O}$ is not computable since the stationary/critical points are not 
\syo{analytically} solvable. On the other hand, our proposed remainder-form inclusion function, $T^{f_d}_{R}$, can be tractably computed. \syo{Figure} \ref{fig:ex_jaulin} depicts the resulting reachable sets, where we observe that $T^{f_d}_{R}$ outperforms all other applicable and computable approaches, and \mo{is} comparable to the ``best of all"  approach (via intersection). 

\begin{figure}[t]
\begin{center}
 \includegraphics[scale=0.2175,trim=33mm 0mm 35mm 0mm,clip]{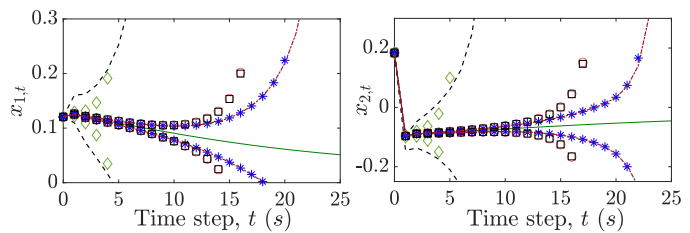}
\vspace{-0.1cm}
\caption{Upper and lower bounds on ${x}_{1}$ and $x_2$ in System \eqref{eq:Jaulin_ex_2_11} (Example C), when applying $T^f_N (- -)$, $T^f_C$ ({\color{red}$\circ$}), $T^f_M$ ({\color{ForestGreen}$\diamond$}), $T^{f_d}_L$ ($\square$), $T^{f_d}_{R}$ ({\color{blue}$\ast$}) and the best of $T^f_N$--$T^{f_d}_{R}$ ({\color{purple}$\cdot$-}), as well as the midpoint trajectory ({\color{ForestGreen}--}). \label{fig:ex_jaulin}}
\end{center}
\vspace{-0.15cm}
\end{figure}

\subsection{Continuous-Time System in \cite{abate2020tight}}
As the next example, we consider the following continuous-time dynamical system from \cite{abate2020tight}:
\begin{align} \label{eq:Coogan_t}
\begin{bmatrix} \dot{x}_{1,t} \\ \dot{x}_{2,t} \\ \dot{x}_{3 ,t} \end{bmatrix}=\begin{bmatrix} w_{1\syo{,t}}x_{2,t}^2-x_{2,t}+w_{2,t} \\x_{3,t}+2 \\ x_{1,t}-x_{2,t}-w_{1,t}^3 \end{bmatrix},
\end{align}
with ${x}_0\in [-\frac{1}{2},\frac{1}{2}]^3$ and $w_t\in [-\frac{1}{4},0] \times [0,\frac{1}{4}]$. Figure \ref{fig:Ex_Coogan_t} depicts the approximations of the reachable sets when applying $T^f_N$, $T^f_C$, $T^f_M$, $T^{f_d}_L$, $T^{f_d}_R$, the best of $T^f_N$--$T^{f_d}_R$ and $T^{f_d}_O$. As expected, the tight $T^{f_d}_O$, which is implemented using the corresponding embedding functions given in \cite[Section VI]{abate2020tight}, returns the tightest sets/framers. Further, among the other inclusion functions, $T^{f_d}_R$ has the the best performance and it \syo{can be} 
slightly improved \syo{by} 
using the best of $T^f_N$--$T^{f_d}_R$.   
\begin{figure}[t]
\vspace{-0.3cm}
\begin{center}
\includegraphics[scale=0.212,trim=66mm 0mm 20mm 0mm,clip]{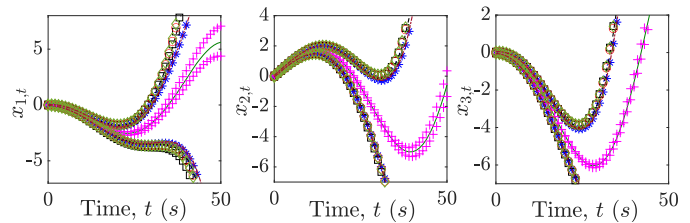}
\vspace{-0.1cm}
\caption{Upper and lower bounds for ${x}_{1}$, $x_2$ and $x_3$ in System \eqref{eq:Coogan_t} (Example D), when applying $T^f_N (- -)$, $T^f_C$ ({\color{red}$\circ$}), $T^f_M$ ({\color{ForestGreen}$\diamond$}), $T^{f_d}_L$ ($\square$), $T^{f_d}_{R}$ ({\color{blue}$\ast$}), the best of $T^f_N$--$T^{f_d}_{R}$ ({\color{purple}$\cdot$-}) and $T^{f_d}_{O}$ ({\color{magenta}+}), as well as the midpoint trajectory ({\color{ForestGreen}--}).  \label{fig:Ex_Coogan_t}}
\end{center}
\vspace{-0.15cm}
\end{figure}

\subsection{Unicycle System}
Next, we are interested in computing the reachable sets for a well-known continuous-time system, namely the unicycle-like mobile robot, e.g., in \cite{chen2018nonlinear,jetto1999development}, 
with two driving wheels, mounted on the left and right sides of the robot, with their common axis passing through the center of the robot. The dynamics of such a system can be described as in \cite{jetto1999development}:  
\begin{gather} \label{eq:unicycle}
\begin{array}{c}
\dot{s}_{x,t}=\phi_{\omega,t}\cos \theta_t +w_{x,t}, \
\dot{s}_{y,t}=\phi_{\omega,t}\sin \theta_t +w_{y,t}, \\
\dot{\theta}_t =\phi_{\theta,t}+w_{\theta,t},
\end{array}
\end{gather}
with state $x_t \triangleq [s_{x,t} \ s_{y,t} \ \theta_t]^\top$, where $s_{x,t}$ and $s_{y,t}$ are the $X$ and $Y$ coordinates of the main axis mid-point between the two driving wheels and $\theta_t$ is the angle between the robot forward axis and the $X$-direction, while $\phi_{\omega,t}$ and $\phi_{\theta,t}$ are the displacement and angular velocities of the robot (as known inputs), respectively, and $w_t=[w_{x,t} \ w_{y,t} \ w_{\theta,t}]^\top$ is the process noise vector. Setting $\phi_{\omega,t}=0.3$, $\phi_{\theta,t}=0.15$,  $w_{x,t}=0.2(0.5\rho_{x_{1,t}}-0.3)$, $w_{y,t}=0.2(0.3\rho_{x_{2,t}}-0.2)$ and $w_{\theta,t}=0.2(0.6\rho_{x_{3,t}}-0.4)$, with $\rho_{x_{l,t}} \in [0,1]$ $(l=1,2,3)$ and initial state $x_0=[0.1 \ 0.2 \ 1]^\top$, Figure \ref{fig:unicycle_1} shows the over-approximations of the reachable sets for system \eqref{eq:unicycle}, using the various methods discussed in this paper. As can be observed, $T^{f_d}_O$, which is computable for this example, returns the tightest intervals. Moreover, in some intervals, natural inclusions and their modifications, return tighter bounds than $T^{f_d}_R$. However, by taking their intersection and computing the best of $T^f_N$--$T^{f_d}_R$, we obtain further improvements. 

Then, to illustrate the effectiveness of our proposed set inversion algorithm for continuous-time systems, we consider observations/measurements similar to \cite{chen2018nonlinear}, as follows. In the $X$--$Y$ plane, two known points, denoted as $(s_{x_i} , s_{y_i})$ $(i = 1, 2)$, are chosen as markers/landmarks. Then, the distance from the robot's planar Cartesian coordinates $(s_{x,t}, s_{y,t})$ to each marker $(s_{x_i}, s_{y_i})$ can be expressed as $d_{i,t}=\sqrt{(s_{x_i}-s_{x,t})^2+(s_{y_i}-s_{y,t})^2}$. Furthermore, the azimuth $\phi_{i,t}$ at time $t$ can be related to the current system state variables $s_{x,t}, s_{y,t}$ and $\theta_t$ as $\phi_{i,t}=\theta_t-\arctan(\frac{s_{y_i}-s_{y,t}}{s_{x_i}-s_{x,t}})$. Treating both the distance $d_{i,t}$ and the azimuth $\phi_{i,t}$ as the measurements, as well as considering measurement disturbances\syo{/noise} $v_t$, the \syo{nonsmooth,} nonlinear observation/measurement equation can be written as:
\begin{align}\label{eq:observation_unicy}
y_t=[d_{1,t} \ \phi_{1,t} \ d_{2,t} \ \phi_{2,t}]^\top +v_t,
\end{align}
with 
$v_{1,t}=0.02 \rho_{{y_1},t}-0.01$, $v_{2,t}=0.03 \rho_{{y_2},t}-0.01$, $v_{3,t}=0.03 \rho_{{y_3},t}-0.02$, $v_{4,t}=0.05 \rho_{{y_4},t}-0.03$ and $\rho_{{y_k},t} \in [0,1]$ $(k=1,2,3,4)$. Now, applying all the methods along with the set inversion approach in Algorithm \ref{algorithm2} to the constrained system \eqref{eq:unicycle}--\eqref{eq:observation_unicy}, one can observe that by taking advantage of the observations, the reachable set approximations have been significantly improved in Figure \ref{fig:unicycle_2} (with observations), when compared to Figure \ref{fig:unicycle_1} (without observations). 
\begin{figure}[t]
\begin{center}
\begin{subfigure}[b]{0.475\textwidth}
\includegraphics[scale=0.218,trim=48mm 0mm 50mm 0mm,clip]{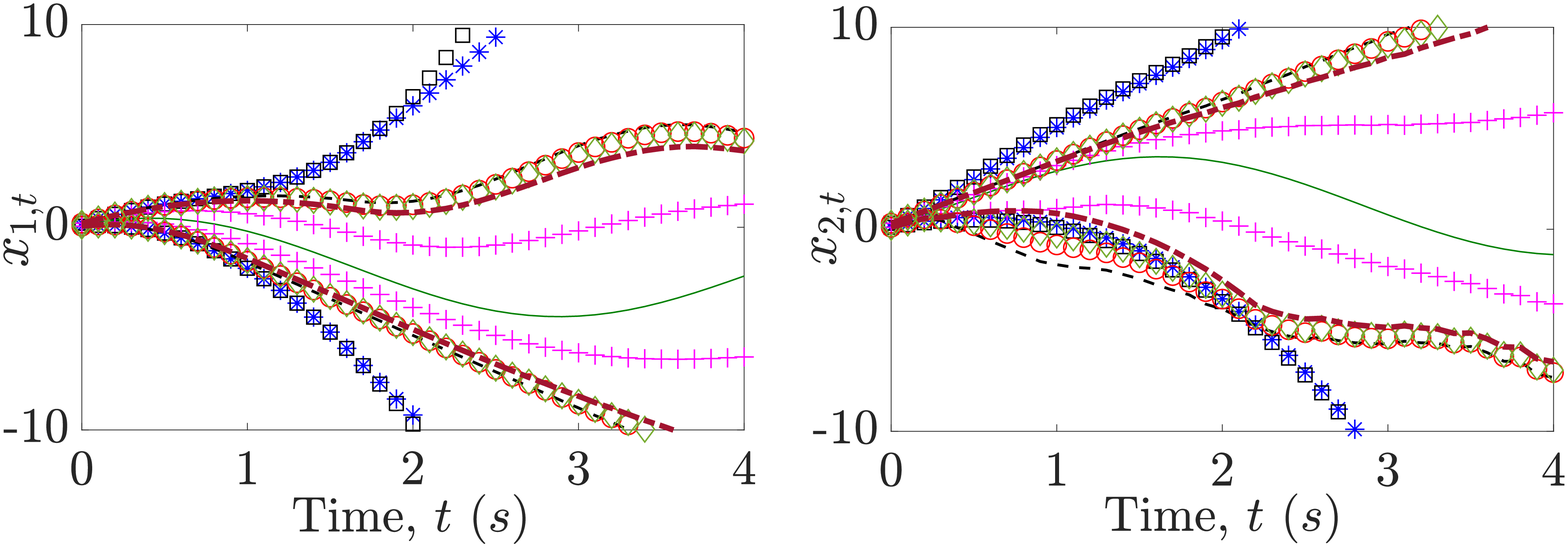}
\vspace{-0.1cm}

\caption{
Unicycle system without observation \eqref{eq:unicycle}.
\label{fig:unicycle_1}}
\end{subfigure}
\begin{subfigure}[b]{0.475\textwidth}
\includegraphics[scale=0.218,trim=48mm 0mm 50mm -10mm,clip]{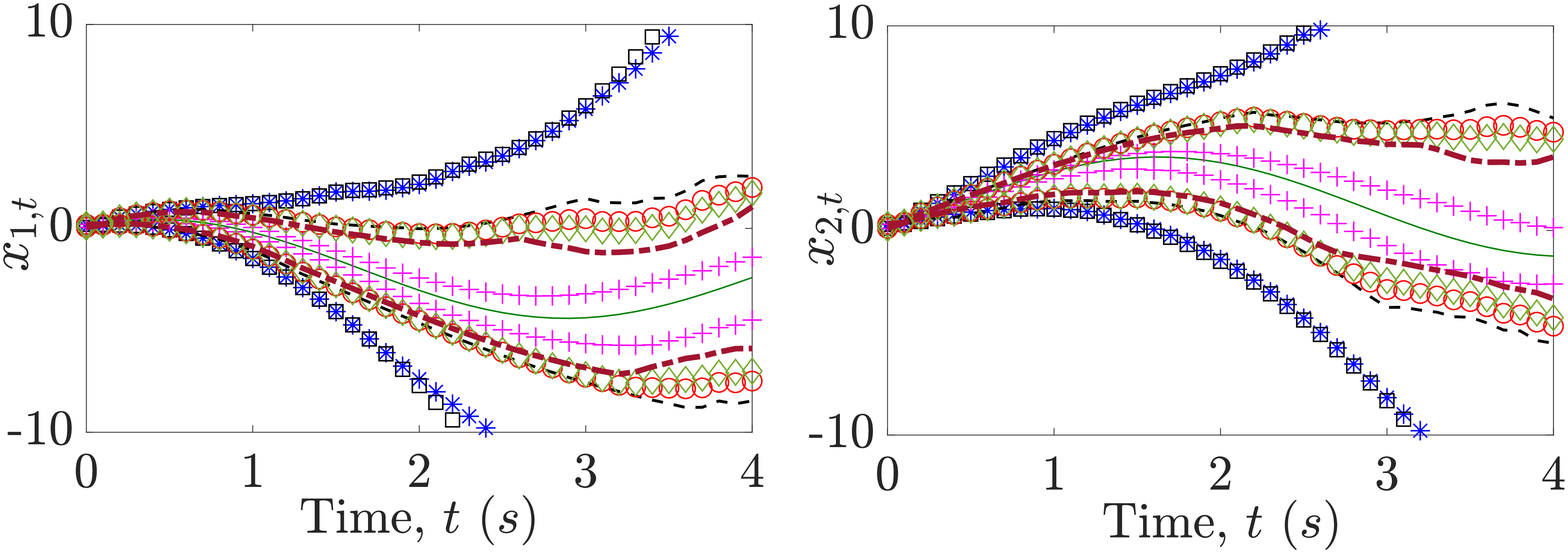} 
\vspace{-0.5cm}
\caption{
Unicycle system with observations \eqref{eq:unicycle}--\eqref{eq:observation_unicy}. 
\label{fig:unicycle_2}}
\end{subfigure}
\end{center}
\vspace{-0.15cm}
\caption{Upper and lower bounds on ${x}_{1}$ and $x_2$ in the unicycle system (Example E), 
when applying $T^f_N (- -)$, $T^f_C$ ({\color{red}$\circ$}), $T^f_M$ ({\color{ForestGreen}$\diamond$}), $T^{f_d}_L$ ($\square$), $T^{f_d}_{R}$ ({\color{blue}$\ast$}), the best of $T^f_N$--$T^{f_d}_{R}$ ({\color{purple}$\cdot$-}) and $T^{f_d}_{O}$ ({\color{magenta}+}), as well as the midpoint trajectory ({\color{ForestGreen}--}). \label{fig:unicycle_all}}
\vspace{-0.15cm}
\end{figure}

\begin{figure}[t]
\begin{center}
\begin{subfigure}[b]{0.475\textwidth}
\includegraphics[scale=0.212,trim=39mm 0mm 50mm 0mm,clip]{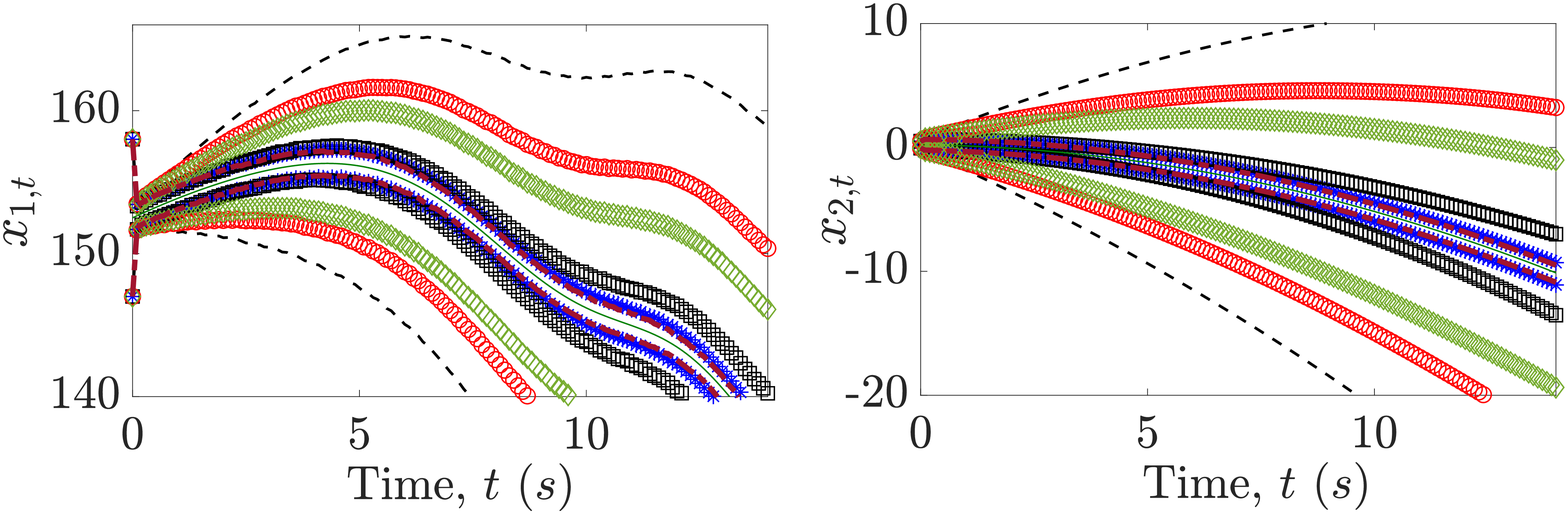}
\vspace{-0.1cm}
\caption{
GTM system without observations \eqref{eq:GTM}
\label{fig:GTLM_1}}
\end{subfigure}
\begin{subfigure}[b]{0.475\textwidth}
\includegraphics[scale=0.212,trim=39mm 0mm 50mm -12mm,clip]{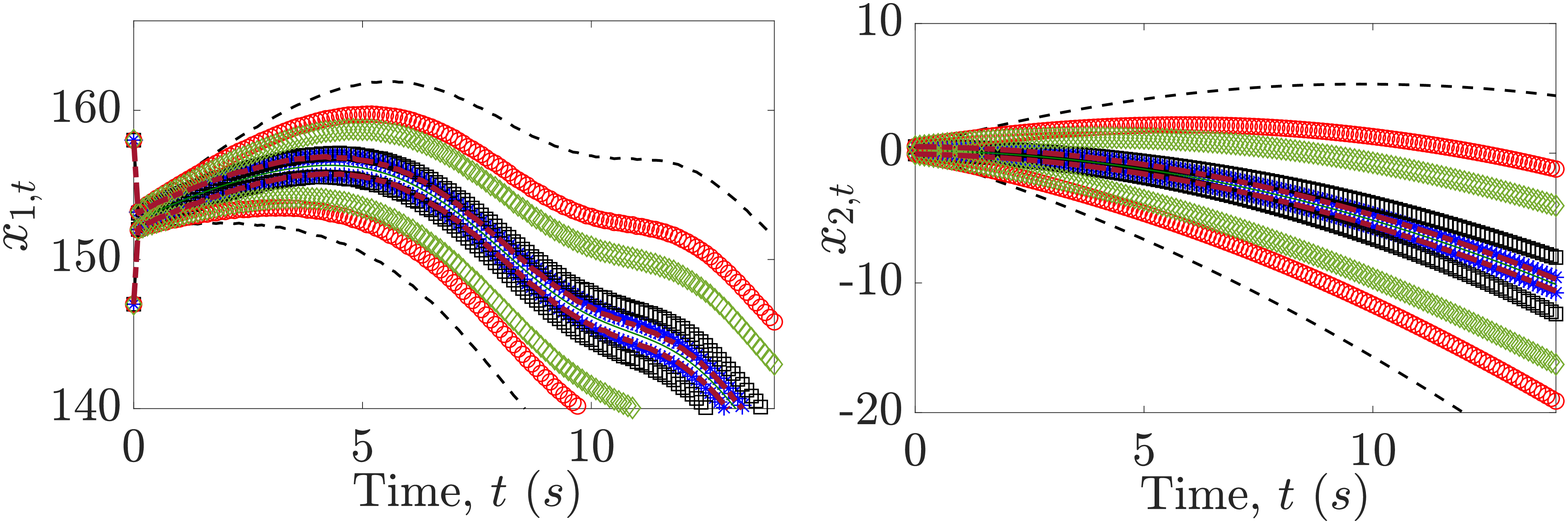}
\vspace{-0.1cm}
\caption{
GTM system with observations \eqref{eq:GTM}--\eqref{eq:GTM_obs}.
\label{fig:GTLM_2}}
\end{subfigure}
\end{center}
\caption{Upper and lower bounds on ${x}_{1}$ and $x_2$ in the GTM system (Example F), when applying $T^f_N (- -)$, $T^f_C$ ({\color{red}$\circ$}), $T^f_M$ ({\color{ForestGreen}$\diamond$}), $T^f_L$ ($\square$), $T^f_{R}$ ({\color{blue}$\ast$}) and the best of $T^f_N$--$T^f_{R}$ ({\color{purple}$\cdot$-}), as well as the midpoint trajectory ({\color{ForestGreen}--}). \label{fig:GTLM_all}}
\vspace{-0.15cm}
\end{figure}
\subsection{Generic Transport Longitudinal Model}
Finally, we consider NASA's Generic Transport Model (GTM) \cite{summers2013quantitative}, 
a remote-controlled 5.5\% scale commercial aircraft \cite{murch2007recent}, with the following main parameters:
wing area $ S= 5.902 \ \text{ft}^2$, mean aerodynamic chord $\overline{c}=0.9153$ ft,
mass $m= 1.542$ slugs, pitch axis moment of inertia $I_{yy}= 4.254 \ \text{slugs/ft}^2$, air density $\rho=0.002375 \ \text{slugs/ft}^3$ and gravitational acceleration $g=32.17 \ \text{ft/s}^2$. The longitudinal dynamics of the GTM can be described as the following continuous-time dynamical system:
\begin{align}\label{eq:GTM}
\begin{array}{ll}
\dot{V}_t&=\frac{-D_t-mg\sin(\theta_t-\alpha_t)+T_{x,t}\cos \alpha_t+T_{z,t}\sin \alpha_t}{m},\\
\dot{\alpha}_t&=q_t+\frac{-L_t+mg\cos(\theta_t-\alpha_t)-T_{x,t}\sin\alpha_t+T_{z,t}\cos\alpha_t}{mV_t},\\
\dot{q}_t&=\frac{M_t+T_{m,t}}{I_{yy}}, \ \dot{\theta}_t=q_t,
\end{array}
\end{align}
where $V_t$, $\alpha_t$, $q_t$ and $\theta_t$ are air speed (ft/s), angle of attack (rad), pitch rate (rad/s) and pitch angle (rad), respectively. Moreover, $T_{x,y}$ (lbs), $T_{z,t}$ (lbs), $T_{m,t}$ (lbs--ft), $D_t$ (lbs), $L_t$ (lbs) and $M_t$ (lbs) denote the projection of the total engine thrust along the body's \syo{$X$}-axis, the projection of the total engine thrust along the body's \syo{$Z$}-axis, the pitching moment due to both engines, the drag force, the lift force and the aerodynamic pitching moment, respectively, with their nominal values given in \cite{stevens2015aircraft}. Defining $x_t \triangleq [V_t \ \alpha_t \ q_t \ \theta_t]^\top$ with ${x}_0\in [147,158] \times [0.04,0.05] \times [0.1,0.2] \times [0.04,0.05]$, Figure \ref{fig:GTLM_1} depicts the reachable set approximations for $x_{1,t}$ and $x_{2,t}$ of the system \eqref{eq:GTM}. For this system\syo{,} $T^{f_d}_O$ is not computable, since the stationary/critical points of the vector fields cannot be obtained \syo{analytically,} 
\mo{a}nd as shown in Figure \ref{fig:GTLM_1}, $T^{f_d}_R$ obtains a tighter over-approximation than $T^f_N,T^f_C$, $T^f_M$ and $T^{f_d}_L$, with the best of $T^f_N$--$T^{f_d}_R$ showing further improvement.

Next, we consider an additional set of measurements, in the form of a linear observation equation: 
\begin{align}\label{eq:GTM_obs}
y_t=x_{1,t}+x_{2,t}-x_{3,t}+v_t,
\end{align} 
with $v_t\in[-0.01,0.01]$. 
Then applying $T^f_N$--$T^{f_d}_R$ along with the set inversion (update) approach in Algorithm \ref{algorithm2} to the constrained system \eqref{eq:GTM}--\eqref{eq:GTM_obs}, we observe considerably tighter approximations for all approaches in Figure \ref{fig:GTLM_2} (with observations) when compared to the approximations of the reachable sets in Figure \ref{fig:GTLM_1} (without observations).

\section{Conclusion And Future Work}

A tractable family of remainder-\syo{form} 
mixed-monotone decomposition functions was proposed in this paper for a relatively large class of nonsmooth systems called either-sided locally Lipschitz systems that is proven to include locally Lipschitz continuous systems. 
We also characterized the lower and upper bounds for the over-approximation errors  when using the proposed remainder-form decomposition functions to over-approximate the true range/image set of a nonsmooth, nonlinear mapping, where the lower bound is achieved by our proposed tight and tractable decomposition function, which was further proven to be tighter than the one introduced in \cite{yang2019sufficient}.
\syo{In addition,} a novel set inversion algorithm based on decomposition functions was developed to further refine/update  the reachabl\mo{e} sets when knowledge of state constraints and/or when measurements are available, which 
 can be applied for constrained reachability analysis and interval observer design for bounded-error continuous-time, discrete-time or sampled-data systems. 
Finally, the effectiveness of our proposed mixed-monotone decomposition functions is demonstrated using several benchmark examples in the literature.
 
Our future work will include the consideration of higher-order information about 
(generalized) derivatives \syo{of functions as well as a more principled way to design modeling redundancy} to \syo{further} improve the tightness of 
decomposition \syo{and inclusion} functions. 
\syo{Moreover,} we plan to extend our proposed interval over-approximation tools to perform reachability analysis with \emph{polytopic} sets, and  
to compute inner- and outer-approximations of reach-avoid \syo{and} 
(controlled) invariant sets. 

\bibliographystyle{IEEEtran}
\bibliography{IEEEabrv,autosam}
\appendices
\end{document}